\documentclass{article}
\usepackage{amssymb}
\usepackage{amsfonts}
\usepackage{amsmath}
\usepackage{geometry}

\newtheorem{theorem}{Theorem}
\newtheorem{acknowledgement}[theorem]{Acknowledgement}

\newtheorem{condition}[theorem]{Condition}

\newtheorem{corollary}[theorem]{Corollary}

\newtheorem{definition}[theorem]{Definition}

\newtheorem{lemma}[theorem]{Lemma}

\newtheorem{proposition}[theorem]{Proposition}
\newtheorem{remark}[theorem]{Remark}

\newenvironment{proof}[1][Proof]{\noindent\textbf{#1.} }{\ \rule{0.5em}{0.5em}}
\input{tcilatex}
\begin{document}

\title{Partial differential equations driven by rough paths\thanks{%
Much of this work has been carried out at the Mittag-Leffler Institute. The
authors would like to express their gratitude to the Institute, as well as
to the organizers of the SPDE\ program (2007) for their invitation. The
first author is partially supported by a Leverhulme Fellowship. Both authors
are supported by EPSRC grant EP/E048609/1.}}
\author{Michael Caruana \and Peter Friz}
\maketitle

\begin{abstract}
We study a class of linear first and second order partial differential
equations driven by weak geometric $p$-rough paths, and prove the existence
of a unique solution for these equations. This solution depends continuously
on the driving rough path. This allows a robust approach to stochastic
partial differential equations.\ In particular, we may replace Brownian
motion by more general Gaussian and Markovian noise. Support theorems and
large deviation statements all became easy corollaries of the corresponding
statements of the driving process. In the case of first order equations with
Gaussian noise, we discuss the existence of a density with respect to the
Lebesgue measure for the solution.
\end{abstract}

\section{Introduction\label{Introduction}}

The theory of rough paths can be described as an extension of the classical
theory of controlled differential equations which is sufficiently robust to
allow a deterministic treatment of stochastic differential equations, and
equations driven by signals which are even more irregular than
semi-martingales. \ Recently various attempts have been made to extend this
theory to partial differential equations (PDEs), with the aim of obtaining
some form of deterministic treatment for stochastic partial differential
equations (SPDEs) and at the same time allowing more general driving signals.

In \cite{gubin-lejay-tindel-2006}, a non-linear evolution problem driven by
a H\"{o}lder continuous path with values in a distribution space is studied.
\ Young integration is used to obtain a mild solution for this equation. \ A
non-linear one-dimensional wave equation driven by signals which satisfy
appropriate H\"{o}lder regularity conditions is considered in \cite%
{quer-tindel-2007}. \ The authors use a 2 dimensional Young integration
theory to solve the wave equation in a mild sense. \ In both these papers, H%
\"{o}lder exponents are assumed to be greater than $\frac{1}{2}$ and
applications to equations driven by Fractional Brownian Motion with Hurst
index greater than $\frac{1}{2}$ are given.

The goal of the present paper is to deal with partial differential equations
of parabolic type of form (with summation over repeated indices) 
\begin{equation}
\frac{\partial u}{\partial t}\left( t,y\right) =\frac{1}{2}a^{ij}\left(
t,y\right) \frac{\partial ^{2}u}{\partial y^{i}\partial y^{j}}+b^{i}\left(
t,y\right) \frac{\partial u}{\partial y^{i}}\left( t,y\right) dt-\frac{%
\partial u}{\partial y^{k}}\left( t,y\right) V_{l}^{k}\left( y\right) \frac{%
dx_{t}^{l}}{dt}
\end{equation}%
with given inital data $u\left( 0,\cdot \right) $, subjected to a
(finite-dimensional) driving signal $\left( x_{t}\right) =\left(
x_{t}^{1},\dots ,x_{t}^{d}\right) $ where $\left( x_{t}\right) $ may only
posses the "rough" regularity of a typical sample path of a stochastic
process; $V_{l}^{k}\left( \cdot \right) $ are sufficiently regular
coefficients. By combining ideas from rough path theory, in particular the
construction of flows associated to \textit{rough differential equations }%
(RDEs) and classical PDE theory we are able to show existence, uniqueness
and a limit theorem for such \textit{rough partial differential equations}
(RPDEs) when the driving signal is a genuine (to be precise: weak,
geometric) $p$-rough path. The main example of such a rough path is given by
(almost every realization of) \textit{Brownian motion and L\'{e}vy's area}
and this allows for a robust treatment of the corresponding classes of
SPDEs. The use of rough path theory in the context of SPDEs has been
conjectued by various people (and in particular by Lyons himself in the
introduction of his '98 article \cite{lyons-98}).\ The present results,
together with those in the just appeared preprint \cite%
{gubinelli-tindel-2008}, seem to be the first steps in this direction. 

This paper is organized as follows. In Section \ref{Preliminaries} we
discuss various concepts we will need from rough path theory, while in
Section \ref{Section RPDEs} we present our results on PDEs driven by weak
geometric rough paths. \ Sections \ref{Section Brownian RPDES}, \ref{Section
Markov RPDEs} and \ref{Section Gaussian RPDEs} are devoted to SPDEs with
multi-dimensional Brownian, Markovian and Gaussian signals\ (Fractional
Brownian Motion, for instance, is covered for $H>\frac{1}{4}$) respectively.
Using the continuity of our solution map, together with results on the
support of the law and large deviation statements for Markovian and Gaussian
rough paths, we get a description of the support of the law of the solution,
and a generalization of the Freidlin Wentzell theorem for these SPDEs. \ In
the case of first order\ equations driven by a class of non-degenerate
Gaussian signals, we also obtain the existence of a density for the
solutions.

\bigskip 

\section{Preliminaries\label{Preliminaries}}

In this section we are going to recall those notions and results from rough
path theory, that will be used in the rest of this paper. \ For a more
complete exposition of this theory, we refer the reader to \cite%
{lyons-qian-2002}, \cite{lyonscarle-2007}, and \cite{FVbook}.\newline
By a smart limiting procedure, ordinary differential equations (ODEs) of
type,%
\begin{equation*}
dy_{t}=\sum_{i}V_{i}\left( y_{t}\right) dx_{t}^{i}\equiv V\left(
y_{t}\right) dx_{t}
\end{equation*}%
defined on the time interval $\left[ 0,T\right] $, started at $y_{0}\in 
\mathbb{R}^{e}$ at time $0$, with Lipschitz vector fields $V=\left(
V_{1},...,V_{d}\right) $ on $\mathbb{R}^{e}$ give rise to so-called rough
differential equations, denoted formally by, 
\begin{equation}
dy_{t}=V\left( y_{t}\right) d\mathbf{x}_{t}  \label{RDE}
\end{equation}%
where $\mathbf{x}$ is  weak geometric $p$-rough path\footnote{%
Strictly speaking we should speak of weak geometric H\"{o}lder $p$-rough
paths.}, that is a $\frac{1}{p}$-H\"{o}lder continuous path from $\left[ 0,T%
\right] $ to $G^{\left[ p\right] }\left( \mathbb{R}^{d}\right) $ (the step-$%
\left[ p\right] $ nilpotent free group over $\mathbb{R}^{d}$), i.e.%
\begin{equation*}
\left\Vert \mathbf{x}_{s,t}\right\Vert \lesssim \left\vert t-s\right\vert ^{%
\frac{1}{p}}\text{ \ for all }s,t\in \left[ 0,T\right] \text{,}
\end{equation*}%
where $\left\Vert \cdot \right\Vert $ is a homogenous norm on $G^{\left[ p%
\right] }\left( \mathbb{R}^{d}\right) $. \ The space of weak geometric H\"{o}%
lder $p$-rough paths is denoted by $C^{\frac{1}{p}-H\ddot{o}l}\left( \left[
0,T\right] ,G^{\left[ p\right] }\left( \mathbb{R}^{d}\right) \right) $, and
for $\mathbf{x}\in C^{\frac{1}{p}-H\ddot{o}l}\left( \left[ 0,T\right] ,G^{%
\left[ p\right] }\left( \mathbb{R}^{d}\right) \right) $, we define,%
\begin{equation*}
\left\Vert \mathbf{x}\right\Vert _{\frac{1}{p}-H\ddot{o}l;\left[ 0,T\right]
}=\sup_{0\leq s<t\leq T}\frac{\left\Vert \mathbf{x}_{s,t}\right\Vert }{%
\left\vert t-s\right\vert ^{\frac{1}{p}}}\text{.}
\end{equation*}%
We also set,%
\begin{equation*}
d_{\frac{1}{p}-H\ddot{o}l;\left[ 0,T\right] }\left( \mathbf{x,\tilde{x}}%
\right) =\sup_{0\leq s<t\leq T}\frac{\left\Vert \mathbf{x}_{s,t}^{-1}\otimes 
\mathbf{\tilde{x}}_{s,t}\right\Vert }{\left\vert t-s\right\vert ^{\frac{1}{p}%
}}\text{.}
\end{equation*}

In the next definition we explain the notion of an RDE solution for (\ref%
{RDE}).

\begin{definition}
\label{solution of rde}Let $\mathbf{x}$ be a weak geometric $p$-rough path,
and suppose that $\left( x^{n}\right) _{n\in \mathbb{N}}$ is a sequence of
Lipschitz paths such that%
\begin{equation*}
S_{\left[ p\right] }\left( x^{n}\right) \equiv \mathbf{x}^{n}\longrightarrow 
\mathbf{x}
\end{equation*}%
uniformly on $\left[ 0,T\right] $ and $\sup_{n}\left\Vert \mathbf{x}%
^{n}\right\Vert _{\frac{1}{p}-H\ddot{o}l;\ \left[ 0,T\right] }<\infty $. \
We call any limit point (in uniform topology on $\left[ 0,T\right] $) of 
\begin{equation*}
\left\{ \pi _{\left( V\right) }\left( 0,y_{0};x^{n}\right) :n\geq 1\right\} 
\end{equation*}%
an RDE solution for (\ref{RDE}) and we denote it by $\pi _{\left( V\right)
}\left( 0,y_{0};\mathbf{x}\right) $. Here,\ $\pi _{\left( V\right) }\left(
0,y_{0};x^{n}\right) $ denotes the solution of the controlled differential
equation, 
\begin{equation*}
dy_{t}^{n}=V\left( y_{t}^{n}\right) dx_{t}^{n}
\end{equation*}%
started at $y_{0}\in \mathbb{R}^{e}$ at time $0$, and $S_{\left[ p\right]
}\left( x^{n}\right) $ is the step-$\left[ p\right] $ signature of $x^{n}$.
\end{definition}

The existence of a sequence of Lipschitz paths $\left( x^{n}\right) _{n\in 
\mathbb{N}}$ with the above properties was established in \cite%
{friz-victoir-2006}. \ In our definition, RDE solutions are genuine $\mathbb{%
R}^{e}$-valued paths. \ It is possible to define RDE solutions as proper
rough paths, but this is of no significance in the present work.

The Universal Limit theorem is one of the main results in rough path theory.
\ It gives a sufficient condition on the vector fields for the existence of
a unique RDE solution, and furthermore, it states that the It\^{o} map which
sends the driving signal to the solution, is continuous.

\begin{theorem}
\label{Thm ULT}Let $\mathbf{x}$ be a weak geometric $p$-rough path and
assume that the vector fields $V=\left( V_{1},...,V_{d}\right) $ are $%
\mathrm{Lip}^{\gamma }\left( \mathbb{R}^{e}\right) $ for $\gamma >p$. \ Then
the RDE%
\begin{equation*}
dy_{t}=V\left( y_{t}\right) d\mathbf{x}_{t}
\end{equation*}%
started at $y_{0}\in \mathbb{R}^{e}$ at time $0$, has a unique RDE solution,
denoted by $\pi _{\left( V\right) }\left( 0,y_{0};\mathbf{x}\right) $. \
Furthermore if $\left( \mathbf{x}_{n}\right) _{n\in \mathbb{N}}\subset C^{%
\frac{1}{p}-H\ddot{o}l}\left( \left[ 0,T\right] ,G^{\left[ p\right] }\left( 
\mathbb{R}^{d}\right) \right) $ converges uniformly to $\mathbf{x}$ with
respect to $d_{\frac{1}{p}-H\ddot{o}l;\left[ 0,T\right] }$ then%
\begin{equation*}
\pi _{\left( V\right) }\left( 0,y_{0};\mathbf{x}^{n}\right) \longrightarrow
\pi _{\left( V\right) }\left( 0,y_{0};\mathbf{x}\right) 
\end{equation*}%
uniformly (in fact, in $1/p$-H\"{o}lder norm).
\end{theorem}

\begin{proof}
c.f. \cite{lyons-qian-2002}, \cite{lyonscarle-2007} and \cite%
{friz-Victoir-2008}.
\end{proof}

One of the elementary operations on rough paths described in \cite%
{lyons-qian-2002}, is time reversal. \ Given $\mathbf{x}\in C^{\frac{1}{p}-H%
\ddot{o}l}\left( \left[ 0,T\right] ,G^{\left[ p\right] }\left( \mathbb{R}%
^{d}\right) \right) $, and a fixed $t\in \left( 0,T\right] $, we can define
a new weak geometric $p$-rough path $\overleftarrow{\mathbf{x}}^{t}$ by 
\begin{equation*}
\overleftarrow{\mathbf{x}}^{t}:\left[ 0,t\right] \longrightarrow G^{\left[ p%
\right] }\left( \mathbb{R}^{d}\right)
\end{equation*}%
\begin{equation*}
\ \ \ \ \ \ \ \ s\ \ \longmapsto \ \ \overleftarrow{\mathbf{x}}_{s}^{t}=%
\mathbf{x}_{t-s}\text{.}
\end{equation*}%
From \cite{lyons-qian-2002}, we know that the map which sends $\mathbf{x}$
to $\overleftarrow{\mathbf{x}}^{t}$, is continuous in $\frac{1}{p}$-H\"{o}%
lder topology.

Our interest in these time reversed paths comes form the following important
fact. \ If we again denote the RDE solution for (\ref{RDE}) by $\pi \left(
0,y;\mathbf{x}\right) $, we have, 
\begin{eqnarray*}
\pi _{\left( V\right) }\left( 0,\pi _{\left( V\right) }\left( 0,y_{0};%
\overleftarrow{\mathbf{x}}^{t}\right) _{t};\mathbf{x}\right) _{t} &=&\pi
_{\left( V\right) }\left( 0,\pi _{\left( V\right) }\left( 0,y_{0};\mathbf{x}%
\right) _{t};\overleftarrow{\mathbf{x}}^{t}\right) _{t} \\
&=&y_{0}\text{.}
\end{eqnarray*}%
Thus,%
\begin{equation*}
\pi _{\left( V\right) }\left( 0,\cdot ;\mathbf{x}\right) _{t}^{-1}\left(
y\right) =\pi _{\left( V\right) }\left( 0,y;\overleftarrow{\mathbf{x}}%
^{t}\right) _{t}
\end{equation*}%
i.e. for each fixed $t$, the inverse of the map $y\longmapsto $ $\pi
_{\left( V\right) }\left( 0,y;\mathbf{x}\right) _{t}$ can be obtained by
solving a rough differential equation driven by the time reversal of the
original driving signal.

The inverse map $\pi _{\left( V\right) }\left( 0,\cdot ;\mathbf{x}\right)
_{t}^{-1}$, and thus $\pi _{\left( V\right) }\left( 0,y;\overleftarrow{%
\mathbf{x}}^{t}\right) _{t}$, the RDE solution for 
\begin{equation*}
dy_{s}=V\left( y_{s}\right) d\overleftarrow{\mathbf{x}}_{s}^{t}
\end{equation*}%
started at $y\in \mathbb{R}^{e}$ at time $0$, will play a very important
role in our definition of a solution for PDEs driven by weak geometric rough
paths.

\section{Rough partial differential equations\label{Section RPDEs}}

Consider partial differential equations of the form%
\begin{eqnarray}
du\left( t,y\right) &=&L_{t}u\left( t,y\right) dt-\sum_{i,j}\partial
_{j}u\left( t,y\right) \cdot V_{i}^{j}\left( y\right) dx_{t}^{i}
\label{2nd order PDE} \\
u\left( 0,y\right) &=&\phi \left( y\right)  \notag
\end{eqnarray}%
on the time interval $\left[ 0,T\right] $, with driving signal $x:\left[ 0,T%
\right] \longrightarrow \mathbb{R}^{d}$, $d$ vector fields $V_{1},\ldots
,V_{d}$ on $\mathbb{R}^{e}$, initial function $\phi :\mathbb{R}%
^{e}\longrightarrow \mathbb{R}$ and $L_{t}$ an elliptic operator of the
form, 
\begin{equation*}
L_{t}=\frac{1}{2}a^{ij}\left( t,\cdot \right) \frac{\partial ^{2}}{\partial
y^{i}\partial y^{j}}+b^{i}\left( t,\cdot \right) \frac{\partial }{\partial
y^{i}}
\end{equation*}%
with\footnote{$S_{e}$ is the set of symmetric non-negative definite $e\times
e$ real matrices.} $a:\left[ 0,T\right] \times \mathbb{R}^{e}\longrightarrow
S_{e}$ and $b:\left[ 0,T\right] \times \mathbb{R}^{e}\longrightarrow \mathbb{%
R}^{e}$. \ In this section we are going to define a notion of a rough
solution for the above PDE when the driving noise is a weak geometric $p$%
-rough path, and then discuss the existence and uniqueness of these
solutions.

Our first task is to define precisely what we mean by a solution for a rough
linear PDE. \ With the definition of an RDE solution (Definition \ref%
{solution of rde})\ in mind, we give the following definition.

\begin{definition}
\label{solution of RPDE}Let $\mathbf{x}$ be a weak geometric $p$-rough path,
and suppose that $\left( x^{n}\right) _{n\in \mathbb{N}}$ is a sequence of
Lipschitz paths such that%
\begin{equation*}
S_{\left[ p\right] }\left( x^{n}\right) \equiv \mathbf{x}^{n}\longrightarrow 
\mathbf{x}
\end{equation*}%
uniformly on $\left[ 0,T\right] $ and $\sup_{n}\left\Vert \mathbf{x}%
^{n}\right\Vert _{\frac{1}{p}-H\ddot{o}l;\ \left[ 0,T\right] }<\infty $. \
Assume that for each $n\in \mathbb{N}$, 
\begin{eqnarray*}
du_{n}\left( t,y\right) &=&L_{t}u_{n}\left( t,y\right) dt-\nabla u_{n}\left(
t,y\right) \cdot V\left( y\right) dx_{t}^{n} \\
u_{n}\left( 0,\cdot \right) &=&\phi \left( \cdot \right) \in C_{b}\left( 
\mathbb{R}^{e}\right)
\end{eqnarray*}%
has a unique $C_{b}^{1,2}$ solution $u_{n}$. \ Then any limit point (in the
uniform topology), of 
\begin{equation*}
\left\{ u_{n}\left( t,y\right) :n\geq 1\right\}
\end{equation*}%
is called a solution for the rough partial differential equation, denoted
formally by,%
\begin{eqnarray}
du\left( t,y\right) &=&L_{t}u\left( t,y\right) dt-\nabla u\left( t,y\right)
\cdot V\left( y\right) d\mathbf{x}_{t}  \label{2nd order RPDE} \\
u\left( 0,y\right) &=&\phi \left( y\right) \text{.}  \notag
\end{eqnarray}
\end{definition}

The rest of this section will be devoted to proving the existence and
uniqueness of solutions for (\ref{2nd order RPDE}). \ The continuity of the
map which sends the driving signal to the solution will also be proved. \ We
will first look at the case $L_{t}\equiv 0$ i.e. we solve a transport
equation driven by a weak geometric $p$-rough path. \ The second order
equation ($L_{t}\neq 0$), which is treated next, can then be seen as a
perturbation of the first order equation.

\subsection{Linear first order RPDEs ($L_{t}\equiv 0$)\label{Subsec 1st
order RPDEs}}

As a motivation for our approach, let us first recall how linear first order
equations are treated in the classical and stochastic cases. \ Consider the
PDE given in (\ref{2nd order PDE}) with $L_{t}\equiv 0$. \ When the path $x:%
\left[ 0,T\right] \longrightarrow \mathbb{R}^{d}$ and the vector fields $%
V_{i}$ are Lipschitz continuous, with an initial function $\phi \in
C^{1}\left( \mathbb{R}^{e},\mathbb{R}\right) $, we can use the method of
characteristics to obtain a unique solution for this equation. \ Indeed, let 
$\pi _{\left( V\right) }\left( 0,y;x\right) $ be the unique solution of the
controlled differential equation, 
\begin{equation*}
dy_{t}=V\left( y_{t}\right) dx_{t}
\end{equation*}%
started at $y\in \mathbb{R}^{e}$ at time $0$. \ Then one can easily show
that for any solution $u$ of 
\begin{eqnarray}
du\left( t,y\right) +V_{i}^{j}\left( y\right) \frac{\partial u\left(
t,y\right) }{\partial y^{j}}dx_{t}^{i} &=&0  \label{PDE} \\
u\left( 0,y\right)  &=&\phi \left( y\right)   \notag
\end{eqnarray}%
we have 
\begin{equation*}
u\left( t,\pi _{\left( V\right) }\left( 0,y;x\right) _{t}\right) =\phi
\left( y\right) \text{.}
\end{equation*}%
Thus we deduce that, 
\begin{eqnarray*}
u\left( t,y\right)  &=&\phi \left( \pi _{\left( V\right) }\left( 0,\cdot
;x\right) _{t}^{-1}\left( y\right) \right)  \\
\ \ \ \ \  &=&\phi \left( \pi _{\left( V\right) }\left( 0,y;\overleftarrow{x}%
^{t}\right) _{t}\right) 
\end{eqnarray*}%
is the unique solution of (\ref{PDE}) with a Lipschitz continuous driving
signal, where $\overleftarrow{x}_{s}^{t}=x_{t-s}$ for $s\in \left[ 0,t\right]
$.

H. Kunita studied first order SPDEs in \cite{Kun-90} using a stochastic
characteristics system, which can be thought of being a generalization of
the method of characteristics to the stochastic case. \ For a first order
linear SPDE driven by a Brownian motion $\left( B_{t}\right) _{t\geq 0}$ in $%
\mathbb{R}^{d}$, 
\begin{eqnarray}
du\left( t,y\right) +V_{i}^{j}\left( y\right) \frac{\partial u\left(
t,y\right) }{\partial y^{j}}\circ dB_{t}^{i} &=&0  \label{SPDE} \\
u\left( 0,y\right)  &=&\phi \left( y\right)   \notag
\end{eqnarray}%
the stochastic characteristic is given by the following Stratonovich SDE, 
\begin{equation}
dy_{t}=V\left( y_{t}\right) \circ dB_{t}~\ \ \ \ \ \ y_{0}=y\in \mathbb{R}%
^{e}\text{.}  \label{characteristic SDE}
\end{equation}%
If the vector fields $V_{i}$ and the initial function $\phi $ are $%
C^{3+\varepsilon }$, then one can use the theory of stochastic flows to
prove that the unique solution of (\ref{SPDE}) is given by, 
\begin{equation*}
u\left( t,y\right) =\phi \left( \pi _{\left( V\right) }\left( 0,\cdot
;B\right) _{t}^{-1}\left( y\right) \right) 
\end{equation*}%
where $\pi _{\left( V\right) }\left( 0,\cdot ;B\right) _{t}$ is the unique
stochastic flow associated with (\ref{characteristic SDE}).

From these brief remarks, we see that the problem of solving first order
linear PDEs with Lipschitz continuous and Brownian signals, can be reduced
to solving an ordinary and stochastic differential equation respectively. \
Therefore a natural question to ask is whether one can use an RDE to solve a
first order linear PDE driven by a weak geometric $p$-rough path.

In the following theorem we give sufficient conditions on the vector fields
and the initial function which guarantee the existence of a unique solution
for a linear first order\ rough PDE driven by a weak geometric $p$-rough
path. \ Moreover, we prove that the map which sends the driving signal to
the solution, is continuous in the uniform topology.

\begin{theorem}
\label{Existence Theorem}Let $p\geq 1$ and let $\mathbf{x}$ be a weak
geometric $p$-rough path. \ Assume that,

\begin{enumerate}
\item $V=\left( V_{1},\ldots ,V_{d}\right) $ is a collection of $\mathrm{Lip}%
^{\gamma }$ vector fields on $\mathbb{R}^{e}$ for $\gamma >p$;

\item $\phi \in C_{b}^{1}\left( \mathbb{R}^{e};\mathbb{R}\right) $.
\end{enumerate}

Then the RPDE, 
\begin{eqnarray}
du\left( t,y\right) +\nabla u\left( t,y\right) \cdot V\left( y\right) d%
\mathbf{x}_{t} &=&0  \label{1st order PDE weak geometric rough path} \\
u\left( 0,y\right)  &=&\phi \left( y\right)   \notag
\end{eqnarray}%
has a unique solution $u$, given explicitly by,%
\begin{equation*}
u\left( t,y\right) =\phi \left( \pi _{\left( V\right) }\left( 0,y;%
\overleftarrow{\mathbf{x}}^{t}\right) _{t}\right) 
\end{equation*}%
where $\pi _{\left( V\right) }\left( 0,y;\mathbf{x}\right) $ was introduced
in Theorem \ref{Thm ULT}. \ We denote the solution $u$\ by $\Pi _{\left(
V\right) }\left( 0,\phi ;\mathbf{x}\right) $. \ Furthermore, the map 
\begin{equation*}
\mathbf{x\longmapsto }u=\Pi _{\left( V\right) }\left( 0,\phi ;\mathbf{x}%
\right) 
\end{equation*}%
is continuous from $C^{\frac{1}{p}-H\ddot{o}l}\left( \left[ 0,T\right] ,G^{%
\left[ p\right] }\left( \mathbb{R}^{d}\right) \right) $ into $C\left( \left[
0,T\right] \times \mathbb{R}^{e}\right) $ when the latter is\ equipped with
the uniform topology.
\end{theorem}

\begin{proof}
Let $\left( x^{n}\right) _{n\in \mathbb{N}}$ be a sequence of Lipschitz
paths such that,%
\begin{equation}
S_{\left[ p\right] }\left( x^{n}\right) \equiv \mathbf{x}^{n}\longrightarrow 
\mathbf{x}  \label{uniform convergence for signature of x(n)}
\end{equation}%
uniformly on $\left[ 0,T\right] $ and, 
\begin{equation*}
\sup_{n}\left\Vert \mathbf{x}^{n}\right\Vert _{\frac{1}{p}-H\ddot{o}l;\left[
0,T\right] }<\infty \text{.}
\end{equation*}%
If we consider the time reversed paths, $\overleftarrow{\mathbf{x}^{n,}}%
_{\cdot }^{t}:=\mathbf{x}_{t-\cdot }^{n}$, we deduce from (\ref{uniform
convergence for signature of x(n)}), that for each fixed $t\in \left( 0,T%
\right] $, 
\begin{equation*}
\overleftarrow{\mathbf{x}^{n,}}_{s}^{t}\longrightarrow \overleftarrow{%
\mathbf{x}}_{s}^{t}
\end{equation*}%
uniformly in $s\in \left[ 0,t\right] $. \ Furthermore, for $0\leq s<u\leq t$,%
\begin{equation*}
\left\vert \overleftarrow{\mathbf{x}^{n,}}_{s,u}^{t}\right\vert =\left\vert 
\mathbf{x}_{t-u,t-s}^{n}\right\vert \leq \left\Vert \mathbf{x}%
^{n}\right\Vert _{\frac{1}{p}-H\ddot{o}l;\left[ 0,T\right] }\left\vert
u-s\right\vert ^{\frac{1}{p}}
\end{equation*}%
and hence, 
\begin{equation}
\sup_{n\in \mathbb{N}}\left\Vert \overleftarrow{\mathbf{x}^{n,}}%
^{t}\right\Vert _{\frac{1}{p}-H\ddot{o}l;\left[ 0,t\right] }\leq \sup_{n\in 
\mathbb{N}}\left\Vert \mathbf{x}^{n}\right\Vert _{\frac{1}{p}-H\ddot{o}l;%
\left[ 0,T\right] }<\infty \text{.}
\label{uniform holder bounds in existence proof}
\end{equation}%
From the Universal Limit Theorem \ref{Thm ULT}, we deduce that $\pi _{\left(
V\right) }\left( 0,y;\overleftarrow{\mathbf{x}^{n,}}^{t}\right) _{s}$
converges uniformly in $s\in \left[ 0,t\right] $ to $\pi _{\left( V\right)
}\left( 0,y;\overleftarrow{\mathbf{x}}^{t}\right) _{s}$, the unique solution
of the RDE, 
\begin{equation}
dy_{s}=V\left( y_{s}\right) d\overleftarrow{\mathbf{x}}_{s}^{t}
\label{RDE existence proof}
\end{equation}%
started at $y\in \mathbb{R}^{e}$ at time $0$. \ In particular we get that 
\begin{equation*}
\pi _{\left( V\right) }\left( 0,y;\overleftarrow{\mathbf{x}^{n,}}^{t}\right)
_{t}\longrightarrow \pi _{\left( V\right) }\left( 0,y;\overleftarrow{\mathbf{%
x}}^{t}\right) _{t}\text{.}
\end{equation*}%
This can of course be done for each $t\in \left[ 0,T\right] $.

Our next task is to prove that the family%
\begin{equation*}
\left\{ \left[ 0,T\right] \times \mathbb{R}^{e}\ni \left( t,y\right)
\longmapsto \pi _{\left( V\right) }\left( 0,y;\overleftarrow{\mathbf{x}^{n,}}%
^{t}\right) _{t}\in \mathbb{R}^{e}\right\} _{n\in \mathbb{N}}
\end{equation*}%
is equicontinuous. \ For $t,t^{\prime }\in \left[ 0,T\right] $ (w.l.o.g $%
t^{\prime }<t$) \ and $y,y^{\prime }\in \mathbb{R}^{e}$,%
\begin{eqnarray}
\left\vert \pi _{\left( V\right) }\left( 0,y;\overleftarrow{\mathbf{x}^{n,}}%
^{t}\right) _{t}-\pi _{\left( V\right) }\left( 0,y^{\prime };\overleftarrow{%
\mathbf{x}^{n,}}^{t^{\prime }}\right) _{t^{\prime }}\right\vert  &\leq
&\left\vert \pi _{\left( V\right) }\left( 0,y;\overleftarrow{\mathbf{x}^{n,}}%
^{t}\right) _{t}-\pi _{\left( V\right) }\left( 0,y;\overleftarrow{\mathbf{x}%
^{n,}}^{t}\right) _{t^{\prime }}\right\vert   \label{equicontinuity ineq} \\
&&+\left\vert \pi _{\left( V\right) }\left( 0,y;\overleftarrow{\mathbf{x}%
^{n,}}^{t}\right) _{t^{\prime }}-\pi _{\left( V\right) }\left( 0,y^{\prime };%
\overleftarrow{\mathbf{x}^{n,}}^{t}\right) _{t^{\prime }}\right\vert 
\label{equicontinuity ineq2} \\
&&+\left\vert \pi _{\left( V\right) }\left( 0,y^{\prime };\overleftarrow{%
\mathbf{x}^{n,}}^{t}\right) _{t^{\prime }}-\pi _{\left( V\right) }\left(
0,y^{\prime };\overleftarrow{\mathbf{x}^{n,}}^{t^{\prime }}\right)
_{t^{\prime }}\right\vert \text{.}  \label{equicontinuity ineq3}
\end{eqnarray}%
From the Generalized Davie Lemma in \cite{friz-Victoir-2008}, we get that,%
\begin{equation*}
\left\vert \pi _{\left( V\right) }\left( 0,y;\overleftarrow{\mathbf{x}^{n,}}%
^{t}\right) _{t}-\pi _{\left( V\right) }\left( 0,y;\overleftarrow{\mathbf{x}%
^{n,}}^{t}\right) _{t^{\prime }}\right\vert \leq C\left\vert t-t^{\prime
}\right\vert ^{\frac{1}{p}}
\end{equation*}%
where the constant $C$ can be chosen to be independent of both $n$ and $t$,
but may depend on $T$ and%
\begin{equation}
A:=\sup_{s\in \left[ 0,T\right] }\sup_{n\in \mathbb{N}}\left\Vert 
\overleftarrow{\mathbf{x}^{n,}}^{s}\right\Vert _{\frac{1}{p}-H\ddot{o}l;%
\left[ 0,s\right] }\leq \sup_{s\in \left[ 0,T\right] }\sup_{n\in \mathbb{N}%
}\left\Vert \mathbf{x}^{n}\right\Vert _{\frac{1}{p}-H\ddot{o}l;\left[ 0,T%
\right] }<\infty \text{.}  \label{uniform in t and n holder bounds}
\end{equation}%
For (\ref{equicontinuity ineq2}) and (\ref{equicontinuity ineq3}), we need
the uniform continuity on $\mathbb{R}^{e}\times \left\{ \mathbf{x:}%
\left\Vert x\right\Vert _{\frac{1}{p}-H\ddot{o}l;\left[ 0,T\right] }\leq
M\right\} $, $M>0$, of the It\^{o} map $\left( y,\mathbf{x}\right)
\longmapsto \pi _{\left( V\right) }\left( 0,y;\mathbf{x}\right) $. \ In
fact, 
\begin{equation*}
\left\vert \pi _{\left( V\right) }\left( 0,y;\overleftarrow{\mathbf{x}^{n,}}%
^{t}\right) _{t^{\prime }}-\pi _{\left( V\right) }\left( 0,y^{\prime };%
\overleftarrow{\mathbf{x}^{n,}}^{t}\right) _{t^{\prime }}\right\vert \leq
\left\vert \pi _{\left( V\right) }\left( 0,y;\overleftarrow{\mathbf{x}^{n,}}%
^{t}\right) -\pi _{\left( V\right) }\left( 0,y^{\prime };\overleftarrow{%
\mathbf{x}^{n,}}^{t}\right) \right\vert _{\infty ;\left[ 0,t^{\prime }\right]
}\longrightarrow 0
\end{equation*}%
uniformly in $n$, as $\left\vert y-y^{\prime }\right\vert \longrightarrow 0$%
, because the uniform bounds in (\ref{uniform in t and n holder bounds})
guarantee that we stay on a bounded set which does not depend on $n$ or $t$.

For (\ref{equicontinuity ineq3})\ we have,%
\begin{equation*}
\left\vert \pi _{\left( V\right) }\left( 0,y^{\prime };\overleftarrow{%
\mathbf{x}^{n,}}^{t}\right) _{t^{\prime }}-\pi _{\left( V\right) }\left(
0,y^{\prime };\overleftarrow{\mathbf{x}^{n,}}^{t^{\prime }}\right)
_{t^{\prime }}\right\vert \leq \left\vert \pi _{\left( V\right) }\left(
0,y^{\prime };\overleftarrow{\mathbf{x}^{n,}}^{t}\right) -\pi _{\left(
V\right) }\left( 0,y^{\prime };\overleftarrow{\mathbf{x}^{n,}}^{t}\right)
\right\vert _{\infty ;\left[ 0,t^{\prime }\right] }\text{.}
\end{equation*}%
Again using the uniform continuity on $\mathbb{R}^{e}\times \left\{ \mathbf{%
x:}\left\Vert x\right\Vert _{\frac{1}{p}-H\ddot{o}l;\left[ 0,T\right] }\leq
M\right\} $ of the It\^{o} map $\left( y,\mathbf{x}\right) \longmapsto \pi
_{\left( V\right) }\left( 0,y;\mathbf{x}\right) $, 
\begin{equation*}
\left\vert \pi _{\left( V\right) }\left( 0,y^{\prime };\overleftarrow{%
\mathbf{x}^{n,}}^{t}\right) -\pi _{\left( V\right) }\left( 0,y^{\prime };%
\overleftarrow{\mathbf{x}^{n,}}^{t}\right) \right\vert _{\infty ;\left[
0,t^{\prime }\right] }\longrightarrow 0
\end{equation*}%
uniformly in $n$, as $\left\vert t-t^{\prime }\right\vert \longrightarrow 0$%
, if we can show that 
\begin{equation}
d_{\frac{1}{p^{\prime }}-H\ddot{o}l;\left[ 0,t^{\prime }\right] }\left( 
\overleftarrow{\mathbf{x}^{n,}}^{t},\overleftarrow{\mathbf{x}^{n,}}%
^{t^{\prime }}\right) \longrightarrow 0  \label{convergence of xnt and xnt'}
\end{equation}%
uniformly in $n$, as $\left\vert t-t^{\prime }\right\vert \longrightarrow 0$%
, for some $p^{\prime }>p$. \ From the interpolation results proved in \cite%
{friz-victoir-2006}, we deduce that (\ref{convergence of xnt and xnt'}) will
follow if we show that%
\begin{equation}
\sup_{s\in \left[ 0,T\right] }\sup_{n\in \mathbb{N}}\left\Vert 
\overleftarrow{\mathbf{x}^{n,}}^{s}\right\Vert _{\frac{1}{p}-H\ddot{o}l;%
\left[ 0,s\right] }<\infty   \label{uniform bounds on xnt}
\end{equation}%
and 
\begin{equation}
d_{0;\left[ 0,t^{\prime }\right] }\left( \overleftarrow{\mathbf{x}^{n,}}^{t},%
\overleftarrow{\mathbf{x}^{n,}}^{t^{\prime }}\right) =\sup_{0\leq s<u\leq
t^{\prime }}d\left( \overleftarrow{\mathbf{x}_{s,u}^{n,}}^{t},\overleftarrow{%
\mathbf{x}_{s,u}^{n,}}^{t^{\prime }}\right) \longrightarrow 0
\label{d0 convergence}
\end{equation}%
uniformly in $n$ as $\left\vert t-t^{\prime }\right\vert \longrightarrow 0$.
\ The required uniform bounds (\ref{uniform bounds on xnt}) are precisely
those obtained in (\ref{uniform in t and n holder bounds}). \ This estimate
guarantees that we stay on a bounded set which does not depend on $n$ or $t$%
. \ In \cite{friz-Victoir-2007a}, the distances $d_{0}$ and $d_{\infty }$
are shown to be locally $\frac{1}{\left[ p\right] }$-H\"{o}lder equivalent,
and hence (\ref{d0 convergence}) will follow if we can show that,%
\begin{equation*}
d_{\infty ;\left[ 0,t^{\prime }\right] }\left( \overleftarrow{\mathbf{x}^{n,}%
}^{t},\overleftarrow{\mathbf{x}^{n,}}^{t^{\prime }}\right) =\sup_{0\leq
s\leq t^{\prime }}d\left( \overleftarrow{\mathbf{x}_{s}^{n,}}^{t},%
\overleftarrow{\mathbf{x}_{s}^{n,}}^{t^{\prime }}\right) \longrightarrow 0
\end{equation*}%
uniformly in $n$, as $\left\vert t-t^{\prime }\right\vert \longrightarrow 0$%
. \ But,%
\begin{eqnarray*}
d_{\infty ;\left[ 0,t^{\prime }\right] }\left( \overleftarrow{\mathbf{x}^{n,}%
}^{t},\overleftarrow{\mathbf{x}^{n,}}^{t^{\prime }}\right)  &=&\sup_{0\leq
s\leq t^{\prime }}d\left( \overleftarrow{\mathbf{x}_{s}^{n,}}^{t},%
\overleftarrow{\mathbf{x}_{s}^{n,}}^{t^{\prime }}\right)  \\
&=&\sup_{0\leq s\leq t^{\prime }}d\left( \mathbf{x}_{t-s}^{n},\mathbf{x}%
_{t^{\prime }-s}^{n}\right)  \\
&\leq &\left\Vert \mathbf{x}^{n}\right\Vert _{\frac{1}{p}-H\ddot{o}l;\left[
0,t^{\prime }\right] }\left\vert t-t^{\prime }\right\vert ^{\frac{1}{p}} \\
&\leq &A\left\vert t-t^{\prime }\right\vert ^{\frac{1}{p}}
\end{eqnarray*}%
and hence, the required convergence (uniform in $n$) is obtained. \
Therefore the family 
\begin{equation*}
\left\{ \left( t,y\right) \longmapsto \pi _{\left( V\right) }\left( 0,y;%
\overleftarrow{\mathbf{x}^{n,}}^{t}\right) _{t}\right\} _{n\in \mathbb{N}}
\end{equation*}%
is indeed equicontinuous in $t$ and $y\in \mathbb{R}^{e}$.

From the pointwise convergence and the equicontinuity, we can conclude that,%
\begin{equation*}
\pi _{\left( V\right) }\left( 0,y;\overleftarrow{\mathbf{x}^{n,}}^{t}\right)
_{t}\longrightarrow \pi _{\left( V\right) }\left( 0,y;\overleftarrow{\mathbf{%
x}}^{t}\right) _{t}
\end{equation*}%
uniformly in $t\in \left[ 0,T\right] $ and $y\in \mathbb{R}^{e}$. \ The
initial function $\phi $ is assumed to be $C_{b}^{1}\left( \mathbb{R}^{e},%
\mathbb{R}\right) $ and hence we get that%
\begin{equation}
\phi \left( \pi _{\left( V\right) }\left( 0,y;\overleftarrow{\mathbf{x}^{n,}}%
^{t}\right) _{t}\right) \longrightarrow \phi \left( \pi _{\left( V\right)
}\left( 0,y;\overleftarrow{\mathbf{x}}^{t}\right) _{t}\right) 
\label{convergence of solutions}
\end{equation}%
uniformly in $t\in \left[ 0,T\right] $ and $y\in \mathbb{R}^{e}$. \
Therefore if we define,%
\begin{equation*}
u\left( t,y\right) =\phi \left( \pi _{\left( V\right) }\left( 0,y;%
\overleftarrow{\mathbf{x}}^{t}\right) _{t}\right) 
\end{equation*}%
we immediately see that $u$ is a solution of (\ref{1st order PDE weak
geometric rough path}).

Having established the existence of a solution of (\ref{1st order PDE weak
geometric rough path}), we now show that a $\mathrm{Lip}^{\gamma }$
assumption on the vector fields guarantees the uniqueness of solutions. \
Suppose that $v:\left[ 0,T\right] \times \mathbb{R}^{e}\longrightarrow 
\mathbb{R}$ is another solution of (\ref{1st order PDE weak geometric rough
path}). \ Then there exists a sequence of Lipschitz paths $z^{n}:\left[ 0,T%
\right] \longrightarrow \mathbb{R}^{d}$ such that, 
\begin{equation*}
S_{\left[ p\right] }\left( z^{n}\right) \equiv \mathbf{z}^{n}\longrightarrow 
\mathbf{x}
\end{equation*}%
and $v\left( t,y\right) =\lim_{n\rightarrow \infty }v_{n}\left( t,y\right) $%
, with $v_{n}$ solving, 
\begin{eqnarray*}
dv_{n}\left( t,y\right) +V_{i}^{j}\left( y\right) \frac{\partial v_{n}\left(
t,y\right) }{\partial y^{j}}dz_{t}^{n,i} &=&0 \\
v_{n}\left( 0,y\right)  &=&\phi \left( y\right) \text{.}
\end{eqnarray*}%
Then, 
\begin{eqnarray*}
v\left( t,y\right)  &=&\lim_{n\rightarrow \infty }v_{n}\left( t,y\right)
=\lim_{n\rightarrow \infty }\phi \left( \pi \left( 0,y;\overleftarrow{%
\mathbf{z}^{n,}}^{t}\right) _{t}\right)  \\
&=&\phi \left( \pi _{\left( V\right) }\left( 0,y;\overleftarrow{\mathbf{x}}%
^{t}\right) _{t}\right)  \\
&=&u\left( t,y\right) 
\end{eqnarray*}%
since $\pi _{\left( V\right) }\left( 0,y;\overleftarrow{\mathbf{z}^{n,}}%
^{t}\right) $ converges to the unique solution $\pi _{\left( V\right)
}\left( 0,y;\overleftarrow{\mathbf{x}}^{t}\right) $ of the RDE (\ref{RDE
existence proof}). \ Therefore the rough solution $u\left( t,y\right) =\phi
\left( \pi _{\left( V\right) }\left( 0,y;\overleftarrow{\mathbf{x}}%
^{t}\right) _{t}\right) $ is indeed unique.

We still have to prove the continuity of the map which sends the driving
signal $\mathbf{x}$ to the solution $u$. \ To this end, suppose that $\left( 
\mathbf{x}^{n}\right) _{n\in \mathbb{N}}$ is a sequence of weak geometric $p$%
-rough paths converging to $\mathbf{x}$ in $\frac{1}{p}$-H\"{o}lder
topology, i.e. $d_{\frac{1}{p}-H\ddot{o}l;\ \left[ 0,T\right] }\left( 
\mathbf{x}^{n},\mathbf{x}\right) \longrightarrow 0$. \ This implies a
fortiori uniform convergence with the uniform bounds $\sup_{n}\left\Vert 
\mathbf{x}^{n}\right\Vert _{\frac{1}{p}-H\ddot{o}l;\left[ 0,T\right]
}<\infty $. \ Using the same reasoning as in the existence part of the
proof, we can show that, 
\begin{equation*}
\pi _{\left( V\right) }\left( 0,y;\overleftarrow{\mathbf{x}^{n,}}^{t}\right)
_{t}\longrightarrow \pi _{\left( V\right) }\left( 0,y;\overleftarrow{\mathbf{%
x}}^{t}\right) _{t}
\end{equation*}%
uniformly in $t\in \left[ 0,T\right] $ and $y\in \mathbb{R}^{e}$. \ Thus,%
\begin{equation*}
u_{n}\left( t,y\right) =\phi \left( \pi _{\left( V\right) }\left( 0,y;%
\overleftarrow{\mathbf{x}^{n,}}^{t}\right) _{t}\right) \longrightarrow \phi
\left( \pi _{\left( V\right) }\left( 0,y;\overleftarrow{\mathbf{x}}%
^{t}\right) _{t}\right) =u\left( t,y\right) 
\end{equation*}%
in $C\left( \left[ 0,T\right] \times \mathbb{R}^{e}\right) $ equipped with
the uniform topology. \ Therefore we conclude that the map which sends the
driving signal to the solution is indeed continuous in the uniform topology.
\end{proof}

\begin{remark}
If we take our initial function $\phi $ to be bounded and uniformly
continuous on $\mathbb{R}^{e}$ i.e. $\phi \in BUC\left( \mathbb{R}%
^{e}\right) $\footnote{$BUC\left( \mathcal{X}\right) $ is the space of
bounded uniformly continuous functions defined on $\mathcal{X}$. \ If $u\in
BUC\left( \mathcal{X}\right) $, then, $\left\Vert u\right\Vert _{BUC\left( 
\mathcal{X}\right) }=\sup_{x\in \mathcal{X}}\left\vert u\left( x\right)
\right\vert _{\mathcal{X}}$.}, then similar reasoning as that used in the
above proof, allows us to conclude that the map 
\begin{equation*}
\mathbf{x}\longrightarrow \Pi _{\left( V\right) }\left( 0,\phi ;\mathbf{x}%
\right) :=\phi \left( \pi _{\left( V\right) }\left( 0,\cdot ;\overleftarrow{%
\mathbf{x}}^{\cdot }\right) _{\cdot }\right) 
\end{equation*}%
is continuous from $C^{\frac{1}{p}-H\ddot{o}l}\left( \left[ 0,T\right] ,G^{%
\left[ p\right] }\left( \mathbb{R}^{d}\right) \right) $ into $BUC\left( %
\left[ 0,T\right] \times \mathbb{R}^{e}\right) $. \ In this case however, $%
\phi \left( \pi _{\left( V\right) }\left( 0,y;\overleftarrow{\mathbf{x}^{n,}}%
^{t}\right) _{t}\right) $ must be interpreted as a weak (e.g. viscosity)\
solution of 
\begin{eqnarray*}
du_{n}\left( t,y\right) +V_{i}^{j}\left( y\right) \frac{\partial u_{n}\left(
t,y\right) }{\partial y^{j}}dx_{t}^{n,i} &=&0 \\
u_{n}\left( 0,y\right)  &=&\phi \left( y\right) \text{.}
\end{eqnarray*}
\end{remark}

\begin{remark}
\label{phi_unbd}If we take $\phi \in C^{1}\left( \mathbb{R}^{e}\right) $,
but not bounded, the map $\mathbf{x}\longrightarrow \Pi _{\left( V\right)
}\left( 0,\phi ;\mathbf{x}\right) $ is continuous in the compact uniform
topology.
\end{remark}

In the next corollary, we show that as in the case of classical and first
order SPDEs, if we assume more regularity on the vector fields and the
initial function, our solution will be smoother in $y$.

\begin{corollary}
\label{regularity of solutions}Let $p\geq 1$, \thinspace $k\in \left\{
1,2,\ldots \right\} $ and let $\mathbf{x}$ be a weak geometric $p$-rough
path. \ Assume that,

\begin{enumerate}
\item $V=\left( V_{1},\ldots ,V_{d}\right) $ is a collection of $\mathrm{Lip}%
^{\gamma }$ vector fields on $\mathbb{R}^{e}$ for $\gamma >p-1+k$;

\item $\phi \in C^{k}\left( \mathbb{R}^{e};\mathbb{R}\right) $.
\end{enumerate}

Then the RPDE 
\begin{eqnarray*}
du\left( t,y\right) +\nabla u\left( t,y\right) \cdot V\left( y\right) d%
\mathbf{x}_{t} &=&0 \\
u\left( 0,y\right) &=&\phi \left( y\right)
\end{eqnarray*}%
has a unique solution $u\in C^{k}\left( \left[ 0,T\right] \times \mathbb{R}%
^{e},\mathbb{R}\right) $.
\end{corollary}

\begin{proof}
From our assumption on the vector fields, we know that for each $t\in \left[
0,T\right] $,%
\begin{equation*}
y\longmapsto \pi _{\left( V\right) }\left( 0,y;\overleftarrow{\mathbf{x}}%
^{t}\right) _{t}
\end{equation*}%
is $k$ times continuously differentiable (c.f. \cite{FVbook}). \ Therefore
we immediately deduce that for each $t\in \left[ 0,T\right] $, 
\begin{equation*}
u\left( t,\cdot \right) =\phi \left( \pi _{\left( V\right) }\left( 0,\cdot ;%
\overleftarrow{\mathbf{x}}^{t}\right) _{t}\right) 
\end{equation*}%
is also $k$ times continuously differentiable.
\end{proof}

\subsection{Second order linear RPDEs\label{Subsec 2nd order RPDEs}}

In what follows we study second order linear PDEs driven by weak geometric $%
p $-rough paths,%
\begin{eqnarray}
du\left( t,y\right) &=&L_{t}u\left( t,y\right) dt-\nabla u\left( t,y\right)
\cdot V\left( y\right) d\mathbf{x}_{t}  \label{2nd order RPDE 2} \\
u\left( 0,y\right) &=&\phi \left( y\right)  \notag
\end{eqnarray}%
where $L_{t}$ is an elliptic operator of the form, 
\begin{equation}
L_{t}=\frac{1}{2}a^{ij}\left( t,\cdot \right) \frac{\partial ^{2}}{\partial
y^{i}\partial y^{j}}+b^{i}\left( t,\cdot \right) \frac{\partial }{\partial
y^{i}}  \label{elliptic operator}
\end{equation}%
with $a:\left[ 0,T\right] \times \mathbb{R}^{e}\longrightarrow S_{e}$ and $b:%
\left[ 0,T\right] \times \mathbb{R}^{e}\longrightarrow \mathbb{R}^{e}$. \
These equations can be regarded as perturbations of first order RPDEs. \ The
approach we use to prove existence and uniqueness of solutions, is based on
a technique for second order linear SPDEs described by Kunita in \cite%
{Kun-90}. \ Kunita shows that solving%
\begin{eqnarray}
du\left( t,y\right) &=&L_{t}u\left( t,y\right) dt-\nabla u\left( t,y\right)
\cdot V\left( y\right) \circ dB_{t}  \label{2nd order SPDE} \\
u\left( 0,y\right) &=&\phi \left( y\right)  \notag
\end{eqnarray}%
can be reduced to proving the existence and uniqueness of solutions for the
following second order PDE,%
\begin{equation}
\frac{\partial v}{\partial t}=\tilde{L}_{t}v\ \ \ \ \ \ v\left( 0,y\right)
=\phi \left( y\right)  \label{perturbed pde in intro}
\end{equation}%
where the coefficients of $\tilde{L}_{t}$ are now random.\footnote{%
These coefficients depend on the flow of the SDE $dy_{t}=V\left(
y_{t}\right) \circ dB_{t}$.}

In what follows we are going to show that these ideas can be generalized to
equations driven by weak geometric $p$-rough paths. \ In our case the PDE
analogue to (\ref{perturbed pde in intro}) will have coefficients which
depend on the flow of an RDE, and so we will sometimes speak of PDEs with
rough coefficients.

Suppose we are given the elliptic operator $L_{t}$,%
\begin{equation*}
L_{t}=\frac{1}{2}a^{ij}\left( t,\cdot \right) \frac{\partial ^{2}}{\partial
y^{i}\partial y^{j}}+b^{i}\left( t,\cdot \right) \frac{\partial }{\partial
y^{i}}
\end{equation*}%
with $a:\left[ 0,T\right] \times \mathbb{R}^{e}\longrightarrow S_{e}$ and $b:%
\left[ 0,T\right] \times \mathbb{R}^{e}\longrightarrow \mathbb{R}^{e}$. \
Let $\mathbf{x}$ be a weak geometric $p$-rough path, $p\geq 1$, and let $%
V=\left( V_{1},\ldots ,V_{d}\right) $ be a collection of $\mathrm{Lip}%
^{\gamma }$ ($\gamma >p+1$)\ vector fields\footnote{%
We need this regularity condition on the vector fields, because to define $%
a_{\mathbf{x}}$ and $b_{\mathbf{x}}$ in (\ref{ax}) and (\ref{bx}) we need a $%
C^{2}$ flow for our RDE (c.f. Corollary \ref{regularity of solutions})} on $%
\mathbb{R}^{e}$. \ For each $t\in \left[ 0,T\right] $, we define the linear
map $w_{t}^{\mathbf{x}}$ on $C^{1}\left( \mathbb{R}^{e}\right) $ by,%
\begin{equation*}
w_{t}^{\mathbf{x}}:C^{1}\left( \mathbb{R}^{e}\right) \longrightarrow
C^{1}\left( \mathbb{R}^{e}\right) 
\end{equation*}%
\begin{equation*}
\ \ \ \ \ \ \ \ \ \phi \longmapsto \Pi _{\left( V\right) }\left( 0,\phi ;%
\mathbf{x}\right) _{t}
\end{equation*}%
i.e. $w_{t}^{\mathbf{x}}\left( \phi \right) $ is the solution of the RPDE (%
\ref{2nd order RPDE 2}) with $L_{t}\equiv 0$. $\ $From Subsection \ref%
{Subsec 1st order RPDEs}, we deduce that$\ w_{t}^{\mathbf{x}}\left( \phi
\right) \left( y\right) =\phi \left( \pi _{\left( V\right) }\left( 0,y;%
\overleftarrow{\mathbf{x}}^{t}\right) _{t}\right) $. \ This map is bijective
and its inverse, $\hat{w}_{t}^{\mathbf{x}}$, is given by,%
\begin{equation*}
\hat{w}_{t}^{\mathbf{x}}\left( \phi \right) \left( y\right) =\phi \left( \pi
_{\left( V\right) }\left( 0,y;\mathbf{x}\right) _{t}\right) \text{.}
\end{equation*}%
We can define a new operator $L_{t}^{\mathbf{x}}$ by,%
\begin{equation*}
L_{t}^{\mathbf{x}}=\hat{w}_{t}^{\mathbf{x}}\circ L_{t}\circ w_{t}^{\mathbf{x}%
}\text{.}
\end{equation*}%
This is again a second order operator represented by, 
\begin{equation*}
L_{t}^{\mathbf{x}}=\frac{1}{2}a_{\mathbf{x}}^{ij}\left( t,\cdot \right) 
\frac{\partial ^{2}}{\partial y^{i}\partial y^{j}}+b_{\mathbf{x}}^{i}\left(
t,\cdot \right) \frac{\partial }{\partial y^{i}}
\end{equation*}%
with 
\begin{equation}
a_{\mathbf{x}}^{ij}\left( t,y\right) =a^{kl}\left( t,\pi _{\left( V\right)
}\left( 0,y;\mathbf{x}\right) _{t}\right) \partial _{k}\pi _{\left( V\right)
}^{i}\left( 0,\cdot ;\overleftarrow{\mathbf{x}}^{t}\right) _{t}|_{\pi
_{\left( V\right) }\left( 0,y;\mathbf{x}\right) _{t}}\partial _{l}\pi
_{\left( V\right) }^{j}\left( 0,\cdot ;\overleftarrow{\mathbf{x}}^{t}\right)
_{t}|_{\pi _{\left( V\right) }\left( 0,y;\mathbf{x}\right) _{t}}  \label{ax}
\end{equation}%
and 
\begin{eqnarray}
b_{\mathbf{x}}^{i}\left( t,y\right)  &=&\frac{1}{2}a^{kl}\left( t,\pi
_{\left( V\right) }\left( 0,y;\mathbf{x}\right) _{t}\right) \partial
_{kl}\pi _{\left( V\right) }^{i}\left( 0,\cdot ;\overleftarrow{\mathbf{x}}%
^{t}\right) _{t}|_{\pi _{\left( V\right) }\left( 0,y;\mathbf{x}\right) _{t}}
\label{bx} \\
&&+b^{k}\left( t,\pi _{\left( V\right) }\left( 0,y;\mathbf{x}\right)
_{t}\right) \partial _{k}\pi _{\left( V\right) }^{i}\left( 0,\cdot ;%
\overleftarrow{\mathbf{x}}^{t}\right) _{t}|_{\pi _{\left( V\right) }\left(
0,y;\mathbf{x}\right) _{t}}\text{.}  \notag
\end{eqnarray}

To better understand why we need the operator $L_{t}^{\mathbf{x}}$, consider
the Lipschitz continuous path $x:\left[ 0,T\right] \longrightarrow \mathbb{R}%
^{d}$. \ Suppose that $u$ is a classical solution of,%
\begin{eqnarray}
du\left( t,y\right) &=&L_{t}u\left( t,y\right) dt-\nabla u\left( t,y\right)
\cdot V\left( y\right) dx_{t}  \label{Lipschitz 2nd RPDE} \\
u\left( 0,y\right) &=&\phi \left( y\right) \in C_{b}^{2}\left( \mathbb{R}%
^{e}\right)  \notag
\end{eqnarray}%
i.e. $u$ is a $C_{b}^{1,2}$ function such that, 
\begin{equation*}
u\left( t,y\right) =\phi \left( y\right) +\int_{0}^{t}L_{s}u\left(
s,y\right) ds-\int_{0}^{t}\nabla u\left( s,y\right) \cdot V\left( y\right)
dx_{s}\text{.}
\end{equation*}%
We then have the following lemma.

\begin{lemma}
\label{Lemma rpde pde}(cf. \cite{Kun-90}) $u$ is a classical solution of (%
\ref{Lipschitz 2nd RPDE}) if and only if $v\left( t,y\right) :=\hat{w}%
_{t}^{x}\left( u\left( t,\cdot \right) \right) \left( y\right) $ is a
classical solution of 
\begin{equation}
\frac{\partial v}{\partial t}=L_{t}^{x}v\ \ \ \ \ \ v\left( 0,y\right) =\phi
\left( y\right) \text{.}  \label{pde in lemma}
\end{equation}
\end{lemma}

\begin{proof}
Let $u$ be a classical solution of (\ref{Lipschitz 2nd RPDE}). \ Then,%
\begin{eqnarray*}
dv\left( t,y\right)  &=&du\left( t,\pi _{\left( V\right) }\left(
0,y;x\right) _{t}\right)  \\
&=&u\left( t,\pi _{\left( V\right) }\left( 0,y;x\right) _{t}\right)
dt+\nabla u\left( t,\pi _{\left( V\right) }\left( 0,y;x\right) _{t}\right)
\cdot d\pi _{\left( V\right) }\left( 0,y;x\right) _{t} \\
&=&u\left( t,\pi _{\left( V\right) }\left( 0,y;x\right) _{t}\right)
dt+\nabla u\left( t,\pi _{\left( V\right) }\left( 0,y;x\right) _{t}\right)
\cdot V\left( \pi _{\left( V\right) }\left( 0,y;x\right) _{t}\right) dx_{t}
\\
&=&L_{t}u\left( t,\pi _{\left( V\right) }\left( 0,y;x\right) _{t}\right) dt
\\
&=&\hat{w}_{t}^{x}\left( L_{t}w_{t}^{x}\left( v\left( t,\cdot \right)
\right) \right) \left( y\right) dt
\end{eqnarray*}%
and therefore $v\left( t,y\right) $ satisfies (\ref{pde in lemma}).

Conversely, we can show that if $v$ is a $C_{b}^{1,2}$\ solution of (\ref%
{pde in lemma}), then 
\begin{equation*}
u\left( t,y\right) :=w_{t}^{x}\left( v\left( t,\cdot \right) \right) \left(
y\right)
\end{equation*}%
is a $C_{b}^{1,2}$ solution of (\ref{Lipschitz 2nd RPDE}).
\end{proof}

Recall that in Definition \ref{solution of RPDE}, we defined a solution of
the RPDE (\ref{2nd order RPDE 2}) to be a limit point of a sequence of
solutions of equations driven by Lipschitz paths converging to $\mathbf{x}$
in rough path sense. \ Thus one of the first things that we need to do, is
discuss the conditions on $a$, $b$, the vector fields $V=\left( V_{1},\ldots
,V_{d}\right) $, and the initial function $\phi $, which guarantee the
existence of a unique $C_{b}^{1,2}$-solution for the classical\ PDE (\ref%
{2nd order PDE}). \ To this end, we have the following regularity condition
on $a$ and $b$.

\begin{condition}
\label{Cond regularity a b}$a:\left[ 0,T\right] \times \mathbb{R}%
^{e}\longrightarrow S_{e}$ and $b:\left[ 0,T\right] \times \mathbb{R}%
^{e}\longrightarrow \mathbb{R}^{e}$ are bounded continuous functions such
that

\begin{enumerate}
\item $a$ is uniformly elliptic i.e. there exists $\lambda >0$ such that,%
\begin{equation*}
\left\langle \theta ,a\left( t,y\right) \theta \right\rangle \geq \lambda
\left\vert \theta \right\vert ^{2}
\end{equation*}%
for all $\left( t,y\right) \in \left[ 0,T\right] \times \mathbb{R}^{e}$ and $%
\theta \in \mathbb{R}^{e}$;

\item there exist constants $C_{a,b}>0$ and $0<\beta \leq 1$ such that,%
\begin{equation*}
\left\vert a\left( t,y\right) -a\left( t^{\prime },y^{\prime }\right)
\right\vert +\left\vert b\left( t,y\right) -b\left( t^{\prime },y^{\prime
}\right) \right\vert \leq C_{a,b}\left( \left\vert t-t^{\prime }\right\vert
^{\beta }+\left\vert y-y^{\prime }\right\vert ^{\beta }\right)
\end{equation*}%
for all $\left( t,y\right) ,\left( t^{\prime }y^{\prime }\right) \in \left[
0,T\right] \times \mathbb{R}^{e}$.
\end{enumerate}
\end{condition}

From Theorem 16, Chapter 1 in \cite{friedman-1964} and Theorem 3.1.1 in \cite%
{StVa-79}, \ we know that if $a$ and $b$ satisfy Condition \ref{Cond
regularity a b}, then the PDE%
\begin{equation*}
\frac{\partial v}{\partial t}=L_{t}v\ \ \ \ \ \ v\left( 0,\cdot \right)
=\phi \left( \cdot \right) \in C_{b}\left( \mathbb{R}^{e}\right)
\end{equation*}%
has a unique $C_{b}^{1,2}$ solution. \ 

\begin{proposition}
\label{Proposition Properties of ax and bx}Let $V=\left( V_{1},\ldots
,V_{d}\right) $ be a collection of $\mathrm{Lip}^{\gamma }$ vector fields, $%
\gamma >p+1$, on $\mathbb{R}^{e}$, and\ suppose that $a:\left[ 0,T\right]
\times \mathbb{R}^{e}\longrightarrow S_{e}$ and $b:\left[ 0,T\right] \times 
\mathbb{R}^{e}\longrightarrow \mathbb{R}^{e}$ satisfy the regularity
condition \ref{Cond regularity a b}. Then the functions $a_{\mathbf{x}}$ and 
$b_{\mathbf{x}}$ defined in (\ref{ax}) and (\ref{bx}) respectively, satisfy,

\begin{enumerate}
\item there exist constants $C_{M}>0$ (uniformly on $\left\{ \mathbf{x}%
:\left\Vert \mathbf{x}\right\Vert _{\frac{1}{p}-H\ddot{o}l}\leq M\right\} $)
such that%
\begin{equation*}
\left\vert a_{\mathbf{x}}\left( t,y\right) -a_{\mathbf{x}}\left( t^{\prime
},y\right) \right\vert +\left\vert b_{\mathbf{x}}\left( t,y\right) -b_{%
\mathbf{x}}\left( t^{\prime },y\right) \right\vert \leq C_{M}\left(
\left\vert t-t^{\prime }\right\vert ^{\beta \wedge \frac{1}{p}}+\left\vert
y-y^{\prime }\right\vert ^{\beta \wedge \frac{1}{p}}\right)
\end{equation*}%
for all $\left( t,y\right) ,\left( t^{\prime }y^{\prime }\right) \in \left[
0,T\right] \times \mathbb{R}^{e}$, where $\beta $ is the H\"{o}lder exponent
of $a$ and $b$;

\item $a_{\mathbf{x}}$ is uniformly elliptic and there exists $\Lambda
_{M}>0 $, such that, 
\begin{equation*}
\inf_{\left\{ \mathbf{x}:\left\Vert \mathbf{x}\right\Vert _{\frac{1}{p}-H%
\ddot{o}l}\leq M\right\} }\left\langle \theta ,a_{\mathbf{x}}\left(
t,y\right) \theta \right\rangle \geq \Lambda _{M}\left\vert \theta
\right\vert ^{2}
\end{equation*}%
for all $\left( t,y\right) \in \left[ 0,T\right] \times \mathbb{R}^{e}$ and $%
\theta \in \mathbb{R}^{e}$.
\end{enumerate}

Furthermore, if we assume that $a$ and $b$ have two bounded, continuous
spatial derivatives and $V=\left( V_{1},\ldots ,V_{d}\right) $ are $\mathrm{%
Lip}^{\gamma }$, $\gamma >p+3$, then $a_{\mathbf{x}}\left( t,\cdot \right) $
and $b_{\mathbf{x}}\left( t,\cdot \right) $ are again$\ C_{b}^{2}$ functions
and their $C^{2}$-norms are uniform over $\left\{ \mathbf{x}:\left\Vert 
\mathbf{x}\right\Vert _{\frac{1}{p}-H\ddot{o}l}\leq M\right\} $.
\end{proposition}

\begin{proof}
Remark that 
\begin{equation*}
\pi _{\left( V\right) }\left( 0,y,\mathbf{x}\right) _{t}=y+\pi _{\left( 
\tilde{V}\right) }\left( 0,\cdot ,\mathbf{x}\right) _{t}
\end{equation*}%
where $\tilde{V}=V\left( y+\cdot \right) $ has the same $\mathrm{Lip}%
^{\gamma }$-norm as $V$. \ It follows that estimates for any derivatives are
automatically uniform over $y$. \ For instance (cf. \cite{FVbook}),%
\begin{equation}
\left\vert D\pi _{\left( V\right) }\left( 0,y,\overleftarrow{\mathbf{x}}%
^{t}\right) _{t}\right\vert \leq C_{1}\exp \left( C_{1}\cdot \left(
\left\vert V\right\vert _{\mathrm{Lip}^{\gamma }}\left\Vert \mathbf{x}%
\right\Vert _{p-var}^{p}\right) \right) 
\label{estimate on derivative of inverse flow}
\end{equation}%
where $C_{1}$ is a constant independent of $y$. \ If we iterate this
argument, we can deduce from (\ref{ax}) and (\ref{bx}), and our regularity
assumption on $a$ and $b$, that $a_{\mathbf{x}}\left( t,\cdot \right) $ and $%
b_{\mathbf{x}}\left( t,\cdot \right) $ are again twice differentiable in
space, with bounded derivatives. \ Furthermore, we can also see from (\ref%
{estimate on derivative of inverse flow}), that the $C^{2}$-norms\footnote{%
For $f\in C^{2}\left( \mathbb{R}^{e}\right) $, we define $\left\Vert
f\right\Vert _{C^{2}}=\sum_{0\leq \left\vert \alpha \right\vert \leq
2}\sup_{x}\left\vert D^{\alpha }f\left( x\right) \right\vert $.} $\left\Vert
a_{\mathbf{x}}\left( t,\cdot \right) \right\Vert _{C^{2}}$ and $\left\Vert
b_{\mathbf{x}}\left( t,\cdot \right) \right\Vert _{C^{2}}$ are bounded
uniformly on $\left\{ \mathbf{x}:\left\Vert \mathbf{x}\right\Vert _{\frac{1}{%
p}-Ho\ddot{l}}\leq M\right\} $.

To prove the uniform ellipticity of $a_{\mathbf{x}}$, we first note that by
assumption, there exists $\lambda >0$, such that,%
\begin{equation*}
\left\langle \theta ,a\left( t,y\right) \theta \right\rangle \geq \lambda
\left\vert \theta \right\vert ^{2}
\end{equation*}%
for all $\left( t,y\right) \in \left[ 0,T\right] \times \mathbb{R}^{e}$ and $%
\theta \in \mathbb{R}^{e}$. \ Hence,%
\begin{eqnarray*}
\left\langle \theta ,a_{\mathbf{x}}\left( t,y\right) \theta \right\rangle 
&\geq &\lambda \left\vert D\pi _{\left( V\right) }\left( 0,\cdot ;%
\overleftarrow{\mathbf{x}}^{t}\right) _{t}|_{\pi _{\left( V\right) }\left(
0,y,\mathbf{x}\right) _{t}}\cdot \theta \right\vert ^{2} \\
&\geq &\frac{\lambda }{\left\vert \left( D\pi _{\left( V\right) }\left(
0,\cdot ;\overleftarrow{\mathbf{x}}^{t}\right) _{t}|_{\pi _{\left( V\right)
}\left( 0,y,\mathbf{x}\right) _{t}}\right) ^{-1}\right\vert }\left\vert
\theta \right\vert ^{2}
\end{eqnarray*}%
since%
\begin{eqnarray*}
\left\vert \theta \right\vert  &=&\left\vert \left( D\pi _{\left( V\right)
}\left( 0,\cdot ;\overleftarrow{\mathbf{x}}^{t}\right) _{t}|_{\pi _{\left(
V\right) }\left( 0,y,\mathbf{x}\right) _{t}}\right) ^{-1}\cdot \left( D\pi
_{\left( V\right) }\left( 0,\cdot ;\overleftarrow{\mathbf{x}}^{t}\right)
_{t}|_{\pi _{\left( V\right) }\left( 0,y,\mathbf{x}\right) _{t}}\right)
\cdot \theta \right\vert  \\
&\leq &\left\vert \left( D\pi _{\left( V\right) }\left( 0,\cdot ;%
\overleftarrow{\mathbf{x}}^{t}\right) _{t}|_{\pi _{\left( V\right) }\left(
0,y,\mathbf{x}\right) _{t}}\right) ^{-1}\right\vert \left\vert \left( D\pi
_{\left( V\right) }\left( 0,\cdot ;\overleftarrow{\mathbf{x}}^{t}\right)
_{t}|_{\pi _{\left( V\right) }\left( 0,y,\mathbf{x}\right) _{t}}\right)
\cdot \theta \right\vert \text{.}
\end{eqnarray*}%
To obtain the uniform ellipticity, we note that, 
\begin{equation*}
\left( D\pi _{\left( V\right) }\left( 0,\cdot ;\overleftarrow{\mathbf{x}}%
^{t}\right) _{t}|_{\pi _{\left( V\right) }\left( 0,y,\mathbf{x}\right)
_{t}}\right) ^{-1}=D\pi _{\left( V\right) }\left( 0,\cdot ;\mathbf{x}\right)
_{t}|_{\pi _{\left( V\right) }\left( 0,y,\mathbf{x}\right) _{t}}\text{.}
\end{equation*}%
Using the already discussed uniformity (with respect to the starting point)
of the Jacobian and other derivatives of the flow, we see that, 
\begin{equation*}
\left\vert \left( D\pi _{\left( V\right) }\left( 0,\cdot ;\overleftarrow{%
\mathbf{x}}^{t}\right) _{t}|_{\pi _{\left( V\right) }\left( 0,y,\mathbf{x}%
\right) _{t}}\right) ^{-1}\right\vert \leq C_{M}
\end{equation*}%
where the constant $C_{M}$ does not depend on $y$. \ This finishes the proof
of the third part of the proposition, since,%
\begin{equation*}
\inf_{\left\{ \mathbf{x}:\left\Vert \mathbf{x}\right\Vert _{\frac{1}{p}-H%
\ddot{o}l}\leq M\right\} }\left\langle \theta ,a_{\mathbf{x}}\left(
t,y\right) \theta \right\rangle \geq \Lambda _{M}\left\vert \theta
\right\vert ^{2}
\end{equation*}%
with $\Lambda _{M}=\frac{\lambda }{C_{M}}$.

The H\"{o}lder continuity of $a_{\mathbf{x}}$ and $b_{\mathbf{x}}$ can be
deduced from the H\"{o}lder continuity of $a$ and $b$, and estimates on the
derivatives of the flow, similar to those in (\ref{equicontinuity ineq}) and
(\ref{equicontinuity ineq2}) in Theorem \ref{Existence Theorem}.
\end{proof}

Therefore, given a weak geometric $p$-rough path $\mathbf{x}$, and $\mathrm{%
Lip}^{\gamma }$, $\gamma >p+1$, vector fields, we can again deduce from
Theorem 16, Chapter 1 in \cite{friedman-1964} and Theorem 3.1.1 in \cite%
{StVa-79}\ that, the PDE with rough coefficients, 
\begin{equation*}
\frac{\partial v}{\partial t}=L_{t}^{\mathbf{x}}v\ \ \ \ \ \ v\left( 0,\cdot
\right) =\phi \left( \cdot \right) \in C_{b}\left( \mathbb{R}^{e}\right)
\end{equation*}%
has a unique $C_{b}^{1,2}$ solution. \ In particular, we have the following
proposition on classical PDEs of the form (\ref{2nd order PDE}).

\begin{proposition}
\label{Prop Sol of PDE with Lipschitz noise}Let $x:\left[ 0,T\right]
\longrightarrow \mathbb{R}^{d}$ be a Lipschitz continuous path. \ Assume
that,

\begin{enumerate}
\item $V=\left( V_{1},\ldots ,V_{d}\right) $ is a collection of $\mathrm{Lip}%
^{2}$ vector fields on $\mathbb{R}^{e}$;

\item let $L_{t}$ be a second order elliptic operator of the form (\ref%
{elliptic operator}), with coefficients $a:\left[ 0,T\right] \times \mathbb{R%
}^{e}\longrightarrow S_{e}$ and $b:\left[ 0,T\right] \times \mathbb{R}%
^{e}\longrightarrow \mathbb{R}^{e}$ satisfying the regularity condition \ref%
{Cond regularity a b};

\item $\phi \in C_{b}\left( \mathbb{R}^{e},\mathbb{R}\right) $.
\end{enumerate}

Then the PDE,%
\begin{eqnarray*}
du\left( t,y\right) &=&L_{t}u\left( t,y\right) dt-\nabla u\left( t,y\right)
\cdot V\left( y\right) dx_{t} \\
u\left( 0,y\right) &=&\phi \left( y\right)
\end{eqnarray*}%
has a unique $C_{b}^{1,2}$ solution $u$. \ 
\end{proposition}

\begin{proof}
This result follows immediately from the previous comments and Lemma \ref%
{Lemma rpde pde}.
\end{proof}

If we now go back to RPDEs, we see from the remarks after Proposition \ref%
{Proposition Properties of ax and bx} and the result in Lemma \ref{Lemma
rpde pde}, that an obvious candidate for the solution of the RPDE (\ref{2nd
order RPDE 2}) is given by 
\begin{equation}
u\left( t,y\right) =w_{t}^{\mathbf{x}}\left( v\right) \left( y\right)
=v\left( t,\pi _{\left( V\right) }\left( 0,y;\overleftarrow{\mathbf{x}}%
^{t}\right) _{t}\right)  \label{candidate for solution}
\end{equation}%
where $v$ is the unique $C_{b}^{1,2}$ solution of 
\begin{equation*}
\frac{\partial v}{\partial t}=L_{t}^{\mathbf{x}}v\ \ \ \ \ \ v\left( 0,\cdot
\right) =\phi \left( \cdot \right) \in C_{b}\left( \mathbb{R}^{e},\mathbb{R}%
\right)
\end{equation*}%
To be able to show that $u$ is the unique solution for (\ref{2nd order RPDE
2}), we first have to prove two propositions, which we will use to show that 
$u$ is in fact the uniform limit of solutions of classical PDEs.

\begin{proposition}
\label{Convergence of coefficients}Let $V=\left( V_{1},\ldots ,V_{d}\right) $
be a collection of $\mathrm{Lip}^{\gamma }$ vector fields, $\gamma >p+1$, on 
$\mathbb{R}^{e}$ and\ suppose that $a:\left[ 0,T\right] \times \mathbb{R}%
^{e}\longrightarrow S_{e}$ and $b:\left[ 0,T\right] \times \mathbb{R}%
^{e}\longrightarrow \mathbb{R}^{e}$ satisfy the regularity condition \ref%
{Cond regularity a b} and\ let $\left( \mathbf{x}^{n}\right) _{n\in \mathbb{N%
}}$ be a sequence of weak geometric $p$-rough paths converging to a weak
geometric $p$-rough path $\mathbf{x}$ uniformly on $\left[ 0,T\right] $ with
uniform bounds i.e. $\sup_{n}\left\Vert \mathbf{x}^{n}\right\Vert _{\frac{1}{%
p}-H\ddot{o}l}<\infty $. \ Then, 
\begin{equation*}
a_{\mathbf{x}^{n}}\left( t,y\right) \longrightarrow a_{\mathbf{x}}\left(
t,y\right)
\end{equation*}%
and%
\begin{equation*}
b_{\mathbf{x}^{n}}\left( t,y\right) \longrightarrow b_{\mathbf{x}}\left(
t,y\right)
\end{equation*}%
uniformly on $\left[ 0,T\right] \times \mathbb{R}^{e}$.
\end{proposition}

\begin{proof}
To prove that $a_{\mathbf{x}^{n}}\left( t,y\right) \longrightarrow a_{%
\mathbf{x}}\left( t,y\right) $ converges uniformly on $\left[ 0,T\right]
\times \mathbb{R}^{e}$, we are first going to obtain pointwise convergence,
and then show that the family $\left\{ \left( t,y\right) \longmapsto a_{%
\mathbf{x}^{n}}\left( t,y\right) \right\} _{n\in \mathbb{N}}$ is
equicontinuous. \ For fixed $\left( t,y\right) \in \left[ 0,T\right] \times 
\mathbb{R}^{e}$ \footnote{%
In what follows we will use the notation, $\xi \left( t,y\right) =\pi
_{\left( V\right) }\left( 0,y;\mathbf{x}\right) _{t}$ and $\zeta \left(
t,y\right) =\pi _{\left( V\right) }\left( 0,y;\overleftarrow{\mathbf{x}}%
^{t}\right) _{t}$.},%
\begin{eqnarray*}
&&\left\vert a_{\mathbf{x}^{n}}^{ij}\left( t,y\right) -a_{\mathbf{x}%
}^{ij}\left( t,y\right) \right\vert  \\
&\leq &\left\vert a^{kl}\left( t,\xi _{n}\left( t,y\right) \right) \partial
_{k}\zeta _{n}^{i}\left( t,\xi _{n}\left( t,y\right) \right) \partial
_{l}\zeta _{n}^{j}\left( t,\xi _{n}\left( t,y\right) \right) -a^{kl}\left(
t,\xi \left( t,y\right) \right) \partial _{k}\zeta ^{i}\left( t,\xi \left(
t,y\right) \right) \partial _{l}\zeta ^{j}\left( t,\xi \left( t,y\right)
\right) \right\vert  \\
&\leq &\left\vert a^{kl}\left( t,\xi _{n}\left( t,y\right) \right)
-a^{kl}\left( t,\xi \left( t,y\right) \right) \right\vert \left\vert
\partial _{k}\zeta ^{i}\left( t,\xi \left( t,y\right) \right) \partial
_{l}\zeta ^{j}\left( t,\xi \left( t,y\right) \right) \right\vert  \\
&&+\left\vert a^{kl}\left( t,\xi _{n}\left( t,y\right) \right) \right\vert
\left\vert \partial _{k}\zeta _{n}^{i}\left( t,\xi _{n}\left( t,y\right)
\right) \partial _{l}\zeta _{n}^{j}\left( t,\xi _{n}\left( t,y\right)
\right) -\partial _{k}\zeta ^{i}\left( t,\xi \left( t,y\right) \right)
\partial _{l}\zeta ^{j}\left( t,\xi \left( t,y\right) \right) \right\vert  \\
&\leq &C_{a,b}\left\vert \xi _{n}\left( t,y\right) -\xi \left( t,y\right)
\right\vert ^{\beta }\left\vert D\zeta \left( t,\xi \left( t,y\right)
\right) \right\vert ^{2} \\
&&+\left\Vert a\right\Vert _{\infty }\left\vert D\zeta _{n}\left( t,\xi
_{n}\left( t,y\right) \right) -D\zeta \left( t,\xi \left( t,y\right) \right)
\right\vert \left( \left\vert D\zeta _{n}\left( t,\xi _{n}\left( t,y\right)
\right) \right\vert +\left\vert D\zeta \left( t,\xi \left( t,y\right)
\right) \right\vert \right) \text{.}
\end{eqnarray*}%
If we let $n\longrightarrow \infty $, we deduce from (\ref{estimate on
derivative of inverse flow}) and the continuity of the It\^{o} map, that $a_{%
\mathbf{x}^{n}}\left( t,y\right) $ converges pointwise to $a_{\mathbf{x}%
}\left( t,y\right) $.

To prove equicontinuity, take $\left( t,y\right) ,\left( t^{\prime
},y^{\prime }\right) \in \left[ 0,T\right] \times \mathbb{R}^{e}$. \ Then,%
\begin{eqnarray}
&&\left\vert a_{\mathbf{x}^{n}}^{ij}\left( t,y\right) -a_{\mathbf{x}%
^{n}}^{ij}\left( t^{\prime },y^{\prime }\right) \right\vert  \notag \\
&\leq &\left\vert a^{kl}\left( t,\xi _{n}\left( t,y\right) \right) \partial
_{k}\zeta _{n}^{i}\left( t,\xi _{n}\left( t,y\right) \right) \partial
_{l}\zeta _{n}^{j}\left( t,\xi _{n}\left( t,y\right) \right) -a^{kl}\left(
t^{\prime },\xi _{n}\left( t^{\prime },y^{\prime }\right) \right) \partial
_{k}\zeta _{n}^{i}\left( t^{\prime },\xi _{n}\left( t^{\prime },y^{\prime
}\right) \right) \partial _{l}\zeta _{n}^{j}\left( t^{\prime },\xi
_{n}\left( t^{\prime },y^{\prime }\right) \right) \right\vert  \notag \\
&\leq &\left\vert a^{kl}\left( t,\xi _{n}\left( t,y\right) \right)
-a^{kl}\left( t^{\prime },\xi _{n}\left( t^{\prime },y^{\prime }\right)
\right) \right\vert \left\vert \partial _{k}\zeta _{n}^{i}\left( t,\xi
_{n}\left( t,y\right) \right) \partial _{l}\zeta _{n}^{j}\left( t,\xi
_{n}\left( t,y\right) \right) \right\vert  \notag \\
&&+\left\vert a^{kl}\left( t^{\prime },\xi _{n}\left( t^{\prime },y^{\prime
}\right) \right) \right\vert \left\vert \partial _{k}\zeta _{n}^{i}\left(
t,\xi _{n}\left( t,y\right) \right) \partial _{l}\zeta _{n}^{j}\left( t,\xi
_{n}\left( t,y\right) \right) -\partial _{k}\zeta _{n}^{i}\left( t^{\prime
},\xi _{n}\left( t^{\prime },y^{\prime }\right) \right) \partial _{l}\zeta
_{n}^{j}\left( t^{\prime },\xi _{n}\left( t^{\prime },y^{\prime }\right)
\right) \right\vert  \notag \\
&\leq &C_{a,b}\left( \left\vert t-t^{\prime }\right\vert ^{\beta
}+\left\vert \xi _{n}\left( t,y\right) -\xi _{n}\left( t^{\prime },y^{\prime
}\right) \right\vert ^{\beta }\right) \left\vert D\zeta ^{n}\left( t,\xi
_{n}\left( t,y\right) \right) \right\vert ^{2}  \label{first term} \\
&&+\left\Vert a\right\Vert _{\infty }\left( \left\vert D\zeta ^{n}\left(
t,\xi _{n}\left( t,y\right) \right) \right\vert +\left\vert D\zeta
^{n}\left( t^{\prime },\xi _{n}\left( t^{\prime },y^{\prime }\right) \right)
\right\vert \right) \left\vert D\zeta ^{n}\left( t,\xi _{n}\left( t,y\right)
\right) -D\zeta ^{n}\left( t^{\prime },\xi _{n}\left( t^{\prime },y^{\prime
}\right) \right) \right\vert \text{.}  \label{second term}
\end{eqnarray}%
Since $\xi _{n}\left( t,y\right) \longrightarrow \xi \left( t,y\right) $
uniformly on $\left[ 0,T\right] \times \mathbb{R}^{e}$, we deduce that $%
\left\{ \left( t,y\right) \longmapsto \xi _{n}\left( t,y\right) \right\}
_{n\in \mathbb{N}}$ is equicontinuous, and hence we can make (\ref{first
term}) arbitrarily small by taking $\left\vert t-t^{\prime }\right\vert $
and $\left\vert y-y^{\prime }\right\vert $ small enough. \ The term (\ref%
{second term}) can also be made arbitrarily small by taking $t$ close to $%
t^{\prime }$ and $y$ close to $y^{\prime }$, because the family%
\begin{equation}
\left\{ \left( t,y\right) \longmapsto D\zeta _{n}\left( t,y\right) \right\}
_{n\in \mathbb{N}}  \label{family of derv of flows}
\end{equation}%
is also equicontinuous. \ This follows because $D\zeta _{n}$ solves an RDE,
and hence similar reasoning as that in Theorem \ref{Existence Theorem} can
be used to prove the equicontinuity of (\ref{family of derv of flows}). \ 

Therefore we can conclude that 
\begin{equation*}
a_{\mathbf{x}^{n}}\left( t,y\right) \longrightarrow a_{\mathbf{x}}\left(
t,y\right)
\end{equation*}%
uniformly on $\left[ 0,T\right] \times \mathbb{R}^{e}$. \ The uniform
convergence of $b_{\mathbf{x}^{n}}\left( t,y\right) $ to$\ b_{\mathbf{x}%
}\left( t,y\right) $ can be proved using a similar procedure.
\end{proof}

Before proving the second proposition, we recall a result by Oleinik (cf.
Theorem 3.2.4 in \cite{StVa-79}).

\begin{theorem}[Oleinik estimate]
\label{Oleinik}Let $L_{t}$ be an elliptic operator of the form (\ref%
{elliptic operator}) with $a\in C_{b}^{0,2}\left( \left[ 0,T\right] \times 
\mathbb{R}^{e};S_{e}\right) $ and $b\in C_{b}^{0,2}\left( \left[ 0,T\right]
\times \mathbb{R}^{e};\mathbb{R}^{e}\right) $ . \ Given $\phi \in
C_{b}^{2}\left( \mathbb{R}^{e}\right) $ and $g\in C_{b}^{0,2}\left( \left[
0,T\right] \times \mathbb{R}^{e}\right) $, suppose that $f\in
C_{b}^{1,2}\left( \left[ 0,T\right] \times \mathbb{R}^{e}\right) $
satisfies, 
\begin{equation*}
\frac{\partial f}{\partial t}-L_{t}f=g
\end{equation*}%
with $f\left( 0,\cdot \right) =\phi $ . \ If 
\begin{equation*}
f\in C_{b}^{0,2}\left( \left[ 0,T\right] \times \mathbb{R}^{e}\right) \cap
C^{0,4}\left( \left[ 0,T\right] \times \mathbb{R}^{e}\right) 
\end{equation*}%
then $\frac{\partial f}{\partial t}\in C^{0,2}\left( \left[ 0,T\right]
\times \mathbb{R}^{e}\right) $ and there exist constants $A$ and $B$ such
that,%
\begin{equation*}
\sup_{0\leq t\leq T}\left\Vert f\left( t,\cdot \right) \right\Vert
_{C^{2}}\leq A\left( 1+\left\Vert \phi \right\Vert _{C^{2}}\right)
+B\sup_{0\leq t\leq T}\left\Vert g\left( t,\cdot \right) \right\Vert _{C^{2}}%
\text{.}
\end{equation*}
\end{theorem}

Using these estimates, we have the following result.

\begin{proposition}
\label{Convergence of sol of pdes}Suppose that for each $n\in \mathbb{N}$, $%
a_{n}:\left[ 0,T\right] \times \mathbb{R}^{e}\longrightarrow S_{e}$ and $%
b_{n}:\left[ 0,T\right] \times \mathbb{R}^{e}\longrightarrow \mathbb{R}^{e}$
satisfy the regularity condition \ref{Cond regularity a b}, and furthermore
assume that they have continuous bounded first and second order spatial
derivatives which are bounded independently of $n$.

Let $a:\left[ 0,T\right] \times \mathbb{R}^{e}\longrightarrow S_{e}$ and $b:%
\left[ 0,T\right] \times \mathbb{R}^{e}\longrightarrow \mathbb{R}^{e}$
satisfy the regularity condition \ref{Cond regularity a b},\ and suppose
that they have bounded first and second order spatial derivatives. \ Assume
that 
\begin{equation*}
a_{n}\left( t,y\right) \longrightarrow a\left( t,y\right)
\end{equation*}%
and%
\begin{equation*}
b_{n}\left( t,y\right) \longrightarrow b\left( t,y\right)
\end{equation*}%
uniformly on $\left[ 0,T\right] \times \mathbb{R}^{e}$. \ Set%
\begin{equation*}
L_{t}^{n}=\frac{1}{2}a_{n}^{ij}\left( t,\cdot \right) \frac{\partial ^{2}}{%
\partial y^{i}\partial y^{j}}+b_{n}^{i}\left( t,\cdot \right) \frac{\partial 
}{\partial y^{i}}
\end{equation*}%
and%
\begin{equation*}
L_{t}=\frac{1}{2}a^{ij}\left( t,\cdot \right) \frac{\partial ^{2}}{\partial
y^{i}\partial y^{j}}+b^{i}\left( t,\cdot \right) \frac{\partial }{\partial
y^{i}}\text{.}
\end{equation*}

Then if we define $v,v_{n}:\left[ 0,T\right] \times \mathbb{R}%
^{e}\longrightarrow \mathbb{R}$ to be the unique $C_{b}^{1,2}$\ solutions of 
\begin{equation}
\frac{\partial v}{\partial t}=L_{t}v\ \ \ \ \ \ \ v\left( 0,\cdot \right)
=\phi \left( \cdot \right) \in C_{b}^{2}\left( \mathbb{R}^{e}\right)
\label{pdeL}
\end{equation}%
and%
\begin{equation}
\frac{\partial v_{n}}{\partial t}=L_{t}^{n}v_{n}\ \ \ \ \ \ \ v_{n}\left(
0,\cdot \right) =\phi \left( \cdot \right) \in C_{b}^{2}\left( \mathbb{R}%
^{e}\right)  \label{pdeLn}
\end{equation}%
respectively, we have that%
\begin{equation*}
v_{n}\left( t,y\right) \longrightarrow v\left( t,y\right)
\end{equation*}%
uniformly on $\left[ 0,T\right] \times \mathbb{R}^{e}$.
\end{proposition}

\begin{proof}
From Theorems 12 and 16, Chapter 1 in \cite{friedman-1964}, we know that (%
\ref{pdeL}) and (\ref{pdeLn}) have unique $C_{b}^{1,2}$ solutions $v$ and $%
v_{n}$, given by, 
\begin{equation}
v\left( t,y\right) =\int_{\mathbb{R}^{e}}\Gamma \left( t,y;0,z\right) \phi
\left( z\right) dz  \label{expression for v}
\end{equation}%
and 
\begin{equation*}
v_{n}\left( t,y\right) =\int_{\mathbb{R}^{e}}\Gamma _{n}\left(
t,y;0,z\right) \phi \left( z\right) dz
\end{equation*}%
where $\Gamma \left( t,y;0,z\right) $ and $\Gamma _{n}\left( t,y;0,z\right) $
are fundamental solutions of $\frac{\partial v}{\partial t}=L_{t}v$ and $%
\frac{\partial v_{n}}{\partial t}=L_{t}^{n}v_{n}$ respectively. \
Furthermore, since $a$ and $b$ have bounded continuous first and second
order spatial derivatives, we deduce from Proposition \ref{Proposition
Properties of ax and bx} and Theorem 10, Chapter 3 in \cite{friedman-1964}
that $v_{n},v\in C^{0,4}\left( \left[ 0,T\right] \times \mathbb{R}%
^{e}\right) $. \ Thus it follows from Theorem \ref{Oleinik} that,%
\begin{equation*}
\sup_{0\leq t\leq T}\left\Vert v_{n}\left( t,\cdot \right) \right\Vert
_{C^{2}}\leq K_{1}\left( 1+\left\Vert \phi \right\Vert _{C^{2}}\right)
\end{equation*}%
where the constant $K_{1}$ can be taken to be independent of $n$ because of
assumption on the spatial derivatives of $a_{n}$ and $b_{n}$. \ Then, 
\begin{eqnarray*}
\left\vert \frac{\partial v_{n}}{\partial t}-L_{t}v_{n}\right\vert
&=&\left\vert L_{t}^{n}v_{n}-L_{t}v_{n}\right\vert \\
&\leq &\left( \left\vert a_{n}\left( t,y\right) -a\left( t,y\right)
\right\vert +\left\vert b_{n}\left( t,y\right) -b\left( t,y\right)
\right\vert \right) \left\Vert v_{n}\left( t,\cdot \right) \right\Vert
_{C^{2}} \\
&\leq &K_{1}\left( 1+\left\Vert \phi \right\Vert _{C^{2}}\right) \left(
\left\vert a_{n}\left( t,y\right) -a\left( t,y\right) \right\vert
+\left\vert b_{n}\left( t,y\right) -b\left( t,y\right) \right\vert \right)
\end{eqnarray*}%
and hence 
\begin{equation}
\frac{\partial v_{n}}{\partial t}-L_{t}v_{n}\longrightarrow 0
\label{conv of vnt - Lnvnt}
\end{equation}%
uniformly on $\left[ 0,T\right] \times \mathbb{R}^{e}$. \ Our next task is
to deduce from (\ref{conv of vnt - Lnvnt}) that the sequence $\left\{
v_{n}\right\} $ converges uniformly. \ To do this, recall (Theorem 12 in 
\cite{friedman-1964}) that under a local H\"{o}lder continuity assumption on
the function $g$,%
\begin{equation}
\hat{v}\left( t,y\right) =\int_{\mathbb{R}^{e}}\phi \left( y\right) \Gamma
\left( t,y;0,z\right) dz-\int_{0}^{t}\left( \int_{\mathbb{R}^{e}}g\left(
s,z\right) \Gamma \left( t,y;s,z\right) dz\right)  \label{inhomogenous}
\end{equation}%
solves the inhomogenous PDE,%
\begin{equation*}
\frac{\partial \hat{v}}{\partial t}-L_{t}\hat{v}=g\ \ \ \ \ \ \ \ \ \ \hat{v}%
\left( 0,y\right) =\phi \left( y\right) \text{.}
\end{equation*}%
Trivially, for $v\,_{n,m}:=v_{n}-v_{m}$, we have,%
\begin{equation*}
\frac{\partial v_{n,m}}{\partial t}-L_{t}v_{n,m}=g_{n,m}\ \ 
\end{equation*}%
with $g_{n,m}=\left( \frac{\partial }{\partial t}-L_{t}\right) v_{n,m}\left(
t,y\right) $. \ We can use the representation (\ref{inhomogenous}), together
with \ref{conv of vnt - Lnvnt} to deduce that $\left\{ v_{n}\right\} $
converges uniformly on $\left[ 0,T\right] \times \mathbb{R}^{e}$ to some
function $\tilde{v}$. \ 

The last step in this proof is to show that $\tilde{v}=v$. \ This follows
because if we repeat the above argument with $g_{n}=\left( \frac{\partial }{%
\partial t}-L_{t}\right) v_{n}$, we get that, 
\begin{equation*}
\tilde{v}\left( t,y\right) =\int_{\mathbb{R}^{e}}\Gamma \left(
t,y;0,z\right) \phi \left( z\right) dz
\end{equation*}%
and thus from (\ref{expression for v}) we see that$\ v=\tilde{v}$. \
Therefore we can conclude that, 
\begin{equation*}
v_{n}\left( t,y\right) \longrightarrow v\left( t,y\right)
\end{equation*}%
uniformly on $\left[ 0,T\right] \times \mathbb{R}^{e}$.
\end{proof}

In the following theorem we prove the existence of a unique bounded solution
for a linear second order RPDE. \ Furthermore, we prove that the map which
sends the driving signal to the solution is continuous in the uniform
topology.

\begin{theorem}
\label{Existence Thm 2nd order}Let $p\geq 1$ and let $\mathbf{x}$ be a weak
geometric $p$-rough path. \ Assume that,

\begin{enumerate}
\item $V=\left( V_{1},\ldots ,V_{d}\right) $ is a collection of $\mathrm{Lip}%
^{\gamma }$ vector fields on $\mathbb{R}^{e}$ for $\gamma >p+3$;

\item $a:\left[ 0,T\right] \times \mathbb{R}^{e}\longrightarrow S_{e}$ and $%
b:\left[ 0,T\right] \times \mathbb{R}^{e}\longrightarrow \mathbb{R}^{e}$
satisfy the regularity condition \ref{Cond regularity a b}, and furthermore,
have continuous bounded first and second order spatial derivatives;

\item $\phi \in C_{b}^{2}\left( \mathbb{R}^{e},\mathbb{R}\right) $.
\end{enumerate}

Assume $L_{t}$ is of form (\ref{elliptic operator}) with coefficients $a,b$.
Then the RPDE,%
\begin{eqnarray}
du\left( t,y\right)  &=&L_{t}u\left( t,y\right) dt-\nabla u\left( t,y\right)
\cdot V\left( y\right) d\mathbf{x}_{t}  \label{2ndRPDE} \\
u\left( 0,y\right)  &=&\phi \left( y\right)   \notag
\end{eqnarray}%
has a unique (bounded) solution $u$, given by,%
\begin{equation*}
u\left( t,y\right) =v\left( t,\pi _{\left( V\right) }\left( 0,y;%
\overleftarrow{\mathbf{x}}^{t}\right) _{t}\right) 
\end{equation*}%
where $v$ is the $C_{b}^{1,2}$ solution of 
\begin{equation*}
\frac{\partial v}{\partial t}=L_{t}^{\mathbf{x}}v\ \ \ \ \ v\left( 0,\cdot
\right) =\phi \left( \cdot \right) \text{.}
\end{equation*}%
We denote this solution by $\Pi _{\left( a,b,V\right) }\left( 0,\phi ;%
\mathbf{x}\right) $. \ Furthermore the map,%
\begin{equation*}
\mathbf{x}\longrightarrow u=\Pi _{\left( a,b,V\right) }\left( 0,\phi ;%
\mathbf{x}\right) 
\end{equation*}%
is continuous from $C^{\frac{1}{p}-H\ddot{o}l}\left( \left[ 0,T\right] ,G^{%
\left[ p\right] }\left( \mathbb{R}^{d}\right) \right) $ into $C\left( \left[
0,T\right] \times \mathbb{R}^{e};\mathbb{R}\right) $ when the latter is\
equipped with the uniform topology.
\end{theorem}

\begin{proof}
We note that the \textrm{Lip}$^{\gamma }$, $\gamma >p+3$, condition on the
vector fields guarantees a $C^{4}$ flow for the associated RDE, and hence
the coefficients $a_{\mathbf{x}}$ and $b_{\mathbf{x}}$ will have bounded
continuous first and second order spatial derivatives (cf. Proposition \ref%
{Proposition Properties of ax and bx}). \ Let $\left( x^{n}\right) _{n\in 
\mathbb{N}}$ be a sequence of Lipschitz paths such that,%
\begin{equation*}
S_{\left[ p\right] }\left( x^{n}\right) \equiv \mathbf{x}^{n}\longrightarrow 
\mathbf{x}
\end{equation*}%
uniformly on $\left[ 0,T\right] $ and, 
\begin{equation*}
\sup_{n}\left\Vert \mathbf{x}^{n}\right\Vert _{\frac{1}{p}-H\ddot{o}l;\left[
0,T\right] }<\infty \text{.}
\end{equation*}%
Let $u_{n}$ be the unique bounded\ solution of,%
\begin{eqnarray*}
du_{n}\left( t,y\right) &=&L_{t}u_{n}\left( t,y\right) dt-\nabla u_{n}\left(
t,y\right) \cdot V\left( y\right) dx_{t}^{n} \\
u_{n}\left( 0,y\right) &=&\phi \left( y\right) \text{.}
\end{eqnarray*}%
We know that such a solution exists because from Proposition \ref{Prop Sol
of PDE with Lipschitz noise}. Then, 
\begin{equation*}
u_{n}\left( t,y\right) =v_{n}\left( t,\pi _{\left( V\right) }\left( 0,y;%
\overleftarrow{\mathbf{x}^{n,}}^{t}\right) _{t}\right)
\end{equation*}%
where $v_{n}$ is the unique $C_{b}^{1,2}$\ classical solution of,%
\begin{equation*}
\frac{\partial v_{n}}{\partial t}=L_{t}^{\mathbf{x}^{n}}v_{n}\ \ \ \ \
v_{n}\left( 0,y\right) =\phi \left( y\right) \text{.}
\end{equation*}

We claim that the function $u$ defined by, 
\begin{equation*}
u\left( t,y\right) =v\left( t,\pi _{\left( V\right) }\left( 0,y;%
\overleftarrow{\mathbf{x}}^{t}\right) _{t}\right)
\end{equation*}%
is a solution of (\ref{2ndRPDE}), and hence we have to show that, 
\begin{equation*}
u\left( t,y\right) =\lim_{n\rightarrow \infty }u_{n}\left( t,y\right)
\end{equation*}%
uniformly on $\left[ 0,T\right] \times \mathbb{R}^{e}$ i.e.%
\begin{equation*}
v\left( t,\pi _{\left( V\right) }\left( 0,y;\overleftarrow{\mathbf{x}}%
^{t}\right) _{t}\right) =\lim_{n\rightarrow \infty }v_{n}\left( t,\pi
_{\left( V\right) }\left( 0,y;\overleftarrow{\mathbf{x}^{n,}}^{t}\right)
_{t}\right)
\end{equation*}%
uniformly on $\left[ 0,T\right] \times \mathbb{R}^{e}$.

We note that the \textrm{Lip}$^{\gamma }$, $\gamma >p+3$, condition on the
vector fields guarantees a $C^{4}$ flow for the associated RDE, and hence
the coefficients $a_{\mathbf{x}}$ and $b_{\mathbf{x}}$ will have bounded
continuous first and second order spatial derivatives (cf. Proposition \ref%
{Proposition Properties of ax and bx}).

Our first task is to prove pointwise convergence. \ For fixed $\left(
t,y\right) \in \left[ 0,T\right] \times \mathbb{R}^{e}$,%
\begin{eqnarray*}
\left\vert v_{n}\left( t,\pi _{\left( V\right) }\left( 0,y;\overleftarrow{%
\mathbf{x}^{n,}}^{t}\right) _{t}\right) -v\left( t,\pi _{\left( V\right)
}\left( 0,y;\overleftarrow{\mathbf{x}}^{t}\right) _{t}\right) \right\vert
&\leq &\left\vert v_{n}\left( t,\pi _{\left( V\right) }\left( 0,y;%
\overleftarrow{\mathbf{x}^{n,}}^{t}\right) _{t}\right) -v\left( t,\pi
_{\left( V\right) }\left( 0,y;\overleftarrow{\mathbf{x}^{n,}}^{t}\right)
_{t}\right) \right\vert \\
&&+\left\vert v\left( t,\pi _{\left( V\right) }\left( 0,y;\overleftarrow{%
\mathbf{x}^{n,}}^{t}\right) _{t}\right) -v\left( t,\pi _{\left( V\right)
}\left( 0,y;\overleftarrow{\mathbf{x}}^{t}\right) _{t}\right) \right\vert 
\text{.}
\end{eqnarray*}%
The second term on the right hand side of this inequality can be made
arbitrarily small by taking $n$ large enough since $v\left( t,\cdot \right) $
is continuous, and 
\begin{equation}
\pi _{\left( V\right) }\left( 0,y;\overleftarrow{\mathbf{x}^{n,}}^{t}\right)
_{t}\longrightarrow \pi _{\left( V\right) }\left( 0,y;\overleftarrow{\mathbf{%
x}}^{t}\right) _{t}\text{.}  \label{conv of rdes in 2nd order pf}
\end{equation}%
For the other term in the inequality, we have that, 
\begin{eqnarray}
&&\left\vert v_{n}\left( t,\pi _{\left( V\right) }\left( 0,y;\overleftarrow{%
\mathbf{x}^{n,}}^{t}\right) _{t}\right) -v\left( t,\pi _{\left( V\right)
}\left( 0,y;\overleftarrow{\mathbf{x}^{n,}}^{t}\right) _{t}\right)
\right\vert  \label{ineq in pointwise conv} \\
&\leq &\left\vert v_{n}\left( t,\pi _{\left( V\right) }\left( 0,y;%
\overleftarrow{\mathbf{x}^{n,}}^{t}\right) _{t}\right) -v_{n}\left( t,\pi
_{\left( V\right) }\left( 0,y;\overleftarrow{\mathbf{x}}^{t}\right)
_{t}\right) \right\vert  \notag \\
&&+\left\vert v\left( t,\pi _{\left( V\right) }\left( 0,y;\overleftarrow{%
\mathbf{x}^{n,}}^{t}\right) _{t}\right) -v\left( t,\pi _{\left( V\right)
}\left( 0,y;\overleftarrow{\mathbf{x}}^{t}\right) _{t}\right) \right\vert 
\notag \\
&&+\left\vert v_{n}\left( t,\pi _{\left( V\right) }\left( 0,y;\overleftarrow{%
\mathbf{x}}^{t}\right) _{t}\right) -v\left( t,\pi _{\left( V\right) }\left(
0,y;\overleftarrow{\mathbf{x}}^{t}\right) _{t}\right) \right\vert \text{.} 
\notag
\end{eqnarray}%
From the results in Proposition \ref{Proposition Properties of ax and bx}, \
we see that Oleinik's estimates in Theorem \ref{Oleinik} can be used for $%
v_{n}$ and $v$, to get 
\begin{equation}
\left\vert v_{n}\left( t,\pi _{\left( V\right) }\left( 0,y;\overleftarrow{%
\mathbf{x}^{n,}}^{t}\right) _{t}\right) -v_{n}\left( t,\pi _{\left( V\right)
}\left( 0,y;\overleftarrow{\mathbf{x}}^{t}\right) _{t}\right) \right\vert
\leq K_{1}\left( 1+\left\Vert \phi \right\Vert _{C^{1}}\right) \left\vert
\pi _{\left( V\right) }\left( 0,y;\overleftarrow{\mathbf{x}^{n,}}^{t}\right)
_{t}-\pi _{\left( V\right) }\left( 0,y;\overleftarrow{\mathbf{x}}^{t}\right)
_{t}\right\vert  \label{ineq1 in pointwise conv}
\end{equation}%
and%
\begin{equation}
\left\vert v\left( t,\pi _{\left( V\right) }\left( 0,y;\overleftarrow{%
\mathbf{x}^{n,}}^{t}\right) _{t}\right) -v\left( t,\pi _{\left( V\right)
}\left( 0,y;\overleftarrow{\mathbf{x}}^{t}\right) _{t}\right) \right\vert
\leq K_{2}\left( 1+\left\Vert \phi \right\Vert _{C^{1}}\right) \left\vert
\pi _{\left( V\right) }\left( 0,y;\overleftarrow{\mathbf{x}^{n,}}^{t}\right)
_{t}-\pi _{\left( V\right) }\left( 0,y;\overleftarrow{\mathbf{x}}^{t}\right)
_{t}\right\vert \text{.}  \label{ineq2 in pointwise conv}
\end{equation}%
Hence we deduce from (\ref{conv of rdes in 2nd order pf}) that both (\ref%
{ineq1 in pointwise conv}) and (\ref{ineq2 in pointwise conv}) go to zero as 
$n\longrightarrow \infty $.

\ The remaining term in (\ref{ineq in pointwise conv}) can also be made
arbitrarily small as $n\longrightarrow \infty $ because the convergence
results in Propositions \ref{Convergence of coefficients} and \ref%
{Convergence of sol of pdes} can be used to deduce that$\
v_{n}\longrightarrow v$.

To prove that the family%
\begin{equation*}
\left\{ \left[ 0,T\right] \times \mathbb{R}^{e}\ni \left( t,y\right)
\longmapsto v_{n}\left( t,\pi _{\left( V\right) }\left( 0,y;\overleftarrow{%
\mathbf{x}^{n,}}^{t}\right) _{t}\right) \right\} _{n\in \mathbb{N}}
\end{equation*}%
is equicontinuous, we take $t^{\prime },t\in \left[ 0,T\right] $ (w.l.o.g $%
t^{\prime }<t$) and $y^{\prime },y\in \mathbb{R}^{e}$, and consider, 
\begin{eqnarray}
&&\left\vert v_{n}\left( t,\pi _{\left( V\right) }\left( 0,y;\overleftarrow{%
\mathbf{x}^{n,}}^{t}\right) _{t}\right) -v_{n}\left( t^{\prime },\pi
_{\left( V\right) }\left( 0,y^{\prime };\overleftarrow{\mathbf{x}^{n,}}%
^{t^{\prime }}\right) _{t^{\prime }}\right) \right\vert 
\label{ineq equicontinuity} \\
&\leq &\left\vert v_{n}\left( t,\pi _{\left( V\right) }\left( 0,y;%
\overleftarrow{\mathbf{x}^{n,}}^{t}\right) _{t}\right) -v_{n}\left(
t^{\prime },\pi _{\left( V\right) }\left( 0,y;\overleftarrow{\mathbf{x}^{n,}}%
^{t}\right) _{t}\right) \right\vert   \notag \\
&&+\left\vert v_{n}\left( t^{\prime },\pi _{\left( V\right) }\left( 0,y;%
\overleftarrow{\mathbf{x}^{n,}}^{t}\right) _{t}\right) -v_{n}\left(
t^{\prime },\pi _{\left( V\right) }\left( 0,y^{\prime };\overleftarrow{%
\mathbf{x}^{n,}}^{t^{\prime }}\right) _{t^{\prime }}\right) \right\vert 
\text{.}  \notag
\end{eqnarray}%
For the first term, 
\begin{eqnarray}
\left\vert v_{n}\left( t,\pi _{\left( V\right) }\left( 0,y;\overleftarrow{%
\mathbf{x}^{n,}}^{t}\right) _{t}\right) -v_{n}\left( t^{\prime },\pi
_{\left( V\right) }\left( 0,y;\overleftarrow{\mathbf{x}^{n,}}^{t}\right)
_{t}\right) \right\vert  &\leq &\int_{t^{\prime }}^{t}\left\vert \frac{%
\partial v_{n}}{\partial s}\left( s,\pi _{\left( V\right) }\left( 0,y;%
\overleftarrow{\mathbf{x}^{n,}}^{t}\right) _{t}\right) \right\vert ds  \notag
\\
&=&\int_{t^{\prime }}^{t}\left\vert L_{s}^{n}v_{n}\left( s,\pi _{\left(
V\right) }\left( 0,y;\overleftarrow{\mathbf{x}^{n,}}^{t}\right) _{t}\right)
\right\vert ds  \notag \\
&\leq &K_{3}\left( 1+\left\Vert \phi \right\Vert _{C^{2}}\right) \left\vert
t-t^{\prime }\right\vert   \label{ineq equicontinuity 1}
\end{eqnarray}%
where $K_{3}$ is a constant which does not depend on $n$. \ To get the last
inequality we again use the estimate in Theorem \ref{Oleinik}. \ For the
other term in (\ref{ineq equicontinuity}), 
\begin{equation}
\left\vert v_{n}\left( t^{\prime },\pi _{\left( V\right) }\left( 0,y;%
\overleftarrow{\mathbf{x}^{n,}}^{t}\right) _{t}\right) -v_{n}\left(
t^{\prime },\pi _{\left( V\right) }\left( 0,y^{\prime };\overleftarrow{%
\mathbf{x}^{n,}}^{t\prime }\right) _{t^{\prime }}\right) \right\vert \leq
K_{4}\left( 1+\left\Vert \phi \right\Vert _{C^{1}}\right) \left\vert \pi
_{\left( V\right) }\left( 0,y;\overleftarrow{\mathbf{x}^{n,}}^{t}\right)
_{t}-\pi _{\left( V\right) }\left( 0,y^{\prime };\overleftarrow{\mathbf{x}%
^{n,}}^{t^{\prime }}\right) _{t^{\prime }}\right\vert \text{.}
\label{ineq equicontinuity 2}
\end{equation}%
In Theorem \ref{Existence Theorem}, \ we proved that the family,%
\begin{equation*}
\left\{ \left[ 0,T\right] \times \mathbb{R}^{e}\ni \left( t,y\right)
\longmapsto \pi _{\left( V\right) }\left( 0,y;\overleftarrow{\mathbf{x}^{n,}}%
^{t}\right) _{t}\in \mathbb{R}^{e}\right\} _{n\in \mathbb{N}}
\end{equation*}%
is equicontinuous and hence we deduce from (\ref{ineq equicontinuity 1}) and
(\ref{ineq equicontinuity 2}) that 
\begin{equation*}
\left\{ \left( t,y\right) \longmapsto v_{n}\left( t,\pi _{\left( V\right)
}\left( 0,y;\overleftarrow{\mathbf{x}^{n,}}^{t}\right) _{t}\right) \right\}
_{n\in \mathbb{N}}
\end{equation*}%
is also equicontinuous.

Therefore we can conclude that,%
\begin{equation*}
u\left( t,y\right) =v\left( t,\pi _{\left( V\right) }\left( 0,y;%
\overleftarrow{\mathbf{x}}^{t}\right) _{t}\right)
\end{equation*}%
is indeed a solution of (\ref{2ndRPDE}) .

Having established existence of solutions for (\ref{2ndRPDE}), we now prove
uniqueness. \ However, as in the case of first order equations, this follows
immediately from the pointwise convergence of 
\begin{equation*}
v_{n}\left( t,\pi _{\left( V\right) }\left( 0,y;\overleftarrow{\mathbf{x}%
^{n,}}^{t}\right) _{t}\right) \longrightarrow v\left( t,\pi _{\left(
V\right) }\left( 0,y;\overleftarrow{\mathbf{x}}^{t}\right) _{t}\right)
\end{equation*}%
proved in the first part of the proof.\ 

We still have to prove the continuity of the map which sends the driving
signal $\mathbf{x}$ to the solution $u$. \ To this end, suppose that $\left( 
\mathbf{x}^{n}\right) _{n\in \mathbb{N}}$ is a sequence of weak geometric $p$%
-rough paths converging to $\mathbf{x}$ in $\frac{1}{p}$-H\"{o}lder
topology, i.e. $d_{\frac{1}{p}-H\ddot{o}l;\ \left[ 0,T\right] }\left( 
\mathbf{x}^{n},\mathbf{x}\right) \longrightarrow 0$. \ This implies a
fortiori uniform convergence with the uniform bounds $\sup_{n}\left\Vert 
\mathbf{x}^{n}\right\Vert _{\frac{1}{p}-H\ddot{o}l;\left[ 0,T\right]
}<\infty $. \ Using the same reasoning as in the existence part of the
proof, we can show that, 
\begin{equation*}
v_{n}\left( t,\pi _{\left( V\right) }\left( 0,y;\overleftarrow{\mathbf{x}%
^{n,}}^{t}\right) _{t}\right) \longrightarrow v\left( t,\pi _{\left(
V\right) }\left( 0,y;\overleftarrow{\mathbf{x}}^{t}\right) _{t}\right)
\end{equation*}%
uniformly in $t\in \left[ 0,T\right] $ and $y\in \mathbb{R}^{e}$. \ Thus,%
\begin{equation*}
u_{n}\left( t,y\right) =v_{n}\left( t,\pi _{\left( V\right) }\left( 0,y;%
\overleftarrow{\mathbf{x}^{n,}}^{t}\right) _{t}\right) \longrightarrow
v\left( t,\pi _{\left( V\right) }\left( 0,y;\overleftarrow{\mathbf{x}}%
^{t}\right) _{t}\right) =u\left( t,y\right)
\end{equation*}%
in $C\left( \left[ 0,T\right] \times \mathbb{R}^{e},\mathbb{R}\right) $
equipped with the uniform topology. \ Therefore we conclude that the map
which sends the driving signal to the solution is indeed continuous in the
uniform topology.
\end{proof}

\section{Application to SPDEs\label{Section Brownian RPDES}}

As is well known (\cite{lyons-qian-2002}, \cite{lyonscarle-2007} and \cite%
{FVbook}), Brownian motion in $\mathbb{R}^{d}$, $B=\left( B^{1},\dots
,B^{d}\right) $, can be enhanced with L\'{e}vy's area and a.s. yields a
geometric $p$-rough path, $p\in \left( 2,3\right) $, denoted by $\mathbf{B}%
\left( \omega \right) \in C^{\frac{1}{p}-H\ddot{o}l}\left( \left[ 0,T\right]
,G^{2}\left( \mathbb{R}^{d}\right) \right) $. \ In the rest of this section
we assume that the elliptic operator $L_{t}$ is given by,%
\begin{equation*}
L_{t}=\frac{1}{2}a^{ij}\left( t,\cdot \right) \frac{\partial ^{2}}{\partial
y^{i}\partial y^{j}}+b^{i}\left( t,\cdot \right) \frac{\partial }{\partial
y^{i}}
\end{equation*}%
with $a$ and $b$ satisfying Condition \ref{Cond regularity a b}, and having
bounded continuous first and second order spatial derivatives.

\begin{proposition}
\label{Prop BM 1}Let $V=\left( V_{1},\ldots ,V_{d}\right) $ be a collection
of $\mathrm{Lip}^{\gamma }$ vector fields on $\mathbb{R}^{e}$, $\gamma >5$,
and suppose that $\phi \in C_{b}^{2}\left( \mathbb{R}^{e}\right) $. \ The
RPDE solution $\Pi _{\left( a,b,V\right) }\left( 0,\phi ;\mathbf{B}\right) $%
, to 
\begin{eqnarray*}
du\left( t,y\right)  &=&L_{t}u\left( t,y\right) dt-\nabla u\left( t,y\right)
\cdot V\left( y\right) d\mathbf{B}_{t}\left( \omega \right)  \\
u\left( 0,y\right)  &=&\phi \left( y\right) 
\end{eqnarray*}%
constructed for fixed $\omega $ in a set of full measure, gives a solution $%
u\left( t,y;\omega \right) $ to the Stratonovich SPDE 
\begin{eqnarray}
du\left( t,y\right)  &=&L_{t}u\left( t,y\right) dt-\nabla u\left( t,y\right)
\cdot V\left( y\right) \circ dB_{t}  \label{StratoSPDE} \\
u\left( 0,y\right)  &=&\phi \left( y\right) \text{.}  \notag
\end{eqnarray}
\end{proposition}

\begin{proof}
Let $B\left( n\right) $ denote the piecewise linear approximation to $B$. \
It is clear from Section $6.4$ in \cite{Kun-90},\ that the solution to%
\begin{equation*}
du\left( t,y\right) =L_{t}u\left( t,y\right) dt-\nabla u\left( t,y\right)
\cdot V\left( y\right) \circ dB_{t}\left( n\right)
\end{equation*}%
converges, at least for fixed $t,y$ and in probability, to the Stratonovich
SPDE solution (\ref{StratoSPDE}). At the same time, $S_{2}\left( B\left(
n\right) \right) \rightarrow \mathbf{B}$ a.s. in $C^{\frac{1}{p}-H\ddot{o}%
l}\left( \left[ 0,T\right] ,G^{2}\left( \mathbb{R}^{d}\right) \right) $. By
the continuity result for RPDEs, we see that the solution to 
\begin{eqnarray*}
du\left( t,y\right) &=&L_{t}u\left( t,y\right) dt-\nabla u\left( t,y\right)
\cdot V\left( y\right) d\mathbf{B}_{t}\left( \omega \right) \\
u\left( 0,y\right) &=&\phi \left( y\right)
\end{eqnarray*}%
is (a version of) the solution to the Stratonovich SPDE.
\end{proof}

In the case of SDEs, if we consider different approximations to Brownian
Motion, the solutions of the corresponding ODEs do not always converge to
the solution of the Stratonovich SDE. \ As shown in \cite{IkWa-89}, this
limit solves a Stratonovich SDE with additional drift terms. \ All this has
been studied from the rough path theory point of view in \cite{LeLy-2005}
and \cite{friz-oberhauser-2008}. \ One of the main examples considered in
this paper is the so-called\textit{\ McShane approximation}\footnote{%
Cf. Pg. 392 \cite{IkWa-89} or Section 5.7 in \cite{Kun-90}.} to Brownian
motion in $\mathbb{R}^{2}$. \ From \cite{LeLy-2005, friz-oberhauser-2008},
the step-$2$ signature of these approximations converge in $\frac{1}{p}$-H%
\"{o}lder topology, $p>2$, to a geometric $p$-rough path $\mathbf{\tilde{B}}$%
, which is basically Brownian motion enhanced with an area which is
different from the usual L\'{e}vy area, i.e.%
\begin{equation*}
\mathbf{\tilde{B}}_{t}=\exp \left( B_{t}+A_{t}+\Gamma t\right) 
\end{equation*}%
where $A_{t}$ is the L\'{e}vy area, and $\Gamma =%
\begin{pmatrix}
0 & c \\ 
-c & 0%
\end{pmatrix}%
$ for some $c$ which may be $\neq 0$.

Furthermore, for $\mathrm{Lip}^{2+\varepsilon }(\mathbb{R}^{2})$ vector
fields $V=\left( V_{1},V_{2}\right) $, it is shown in \cite%
{friz-oberhauser-2008} that $y_{t}$ is a solution of 
\begin{equation*}
dy_{t}=V\left( y_{t}\right) d\mathbf{\tilde{B}}_{t}
\end{equation*}%
started at $y_{0}\in \mathbb{R}^{e}$ if and only if, $y_{t}$ solves,%
\begin{equation*}
dy_{t}=V\left( y_{t}\right) d\mathbf{B}_{t}+c\left[ V_{1},V_{2}\right]
\left( y_{t}\right) dt
\end{equation*}%
started at $y_{0}\in \mathbb{R}^{2}$. \ Here $\mathbf{B}$ is the
Stratonovich Enhanced Brownian motion. \ Thus, 
\begin{equation*}
\pi _{\left( V_{1},V_{2}\right) }\left( 0,y_{0};\mathbf{\tilde{B}}\right)
=\pi _{\left( c\left[ V_{1},V_{2}\right] ,V_{1},V_{2}\right) }\left(
0,y_{0};\left( t,\mathbf{B}\right) \right) 
\end{equation*}%
where $\left( t,\mathbf{B}\right) $ is the canonical time-space rough path
associated with $B$. \ With the above in mind, we prove the following result.

\begin{proposition}
Let $V=\left( V_{1},V_{2}\right) $ be $\mathrm{Lip}^{\gamma }$, $\gamma >5$,
vector fields on $\mathbb{R}^{2}$, and suppose that $\phi \in
C_{b}^{2}\left( \mathbb{R}^{2}\right) $. \ Let $B\left( n\right) $ be the
McShane approximation to Brownian motion. \ Then the $C_{b}^{1,2}$ solutions
to 
\begin{eqnarray}
du^{n}\left( t,y\right)  &=&L_{t}u^{n}\left( t,y\right) dt-\nabla
u^{n}\left( t,y\right) \cdot V\left( y\right) \circ dB_{t}\left( n\right) 
\label{McShane RPDE} \\
u^{n}\left( 0,y\right)  &=&\phi \left( y\right)   \notag
\end{eqnarray}%
converge to the solution of the Stratonovich SPDE 
\begin{eqnarray}
dv\left( t,y\right)  &=&\left( L_{t}v\left( t,y\right) -\nabla v\left(
t,y\right) \cdot c\left[ V_{1},V_{2}\right] \left( y\right) \right)
dt-\nabla v\left( t,y\right) \cdot V\left( y\right) \circ dB_{t}
\label{SPDE with drift} \\
v\left( 0,y\right)  &=&\phi \left( y\right) \text{.}  \notag
\end{eqnarray}
\end{proposition}

\begin{proof}
From our continuity result in Theorem \ref{Existence Thm 2nd order}, we know
that 
\begin{equation*}
u^{n}\left( t,y\right) \longrightarrow u\left( t,y\right)
\end{equation*}%
uniformly on $\left[ 0,T\right] \times \mathbb{R}^{2}$, where $u$ is the
unique solution of the RPDE,%
\begin{eqnarray*}
du\left( t,y\right) &=&L_{t}u\left( t,y\right) dt-\nabla u\left( t,y\right)
\cdot V\left( y\right) d\mathbf{\tilde{B}}_{t}\left( \omega \right) \\
u\left( 0,y\right) &=&\phi \left( y\right)
\end{eqnarray*}%
Furthermore, 
\begin{equation}
u\left( t,y\right) =v\left( t,\pi _{\left( V\right) }\left( 0,y;%
\overleftarrow{\mathbf{\tilde{B}}}^{t}\right) _{t}\right)  \label{defn of u}
\end{equation}%
where $v$ is the unique $C_{b}^{1,2}$ solution of 
\begin{equation*}
\frac{\partial v}{\partial t}=L_{t}^{\mathbf{\tilde{B}}}v\ \ \ \ \ v\left(
0,y\right) =\phi \left( y\right) \text{.}
\end{equation*}%
But from the results in \cite{friz-oberhauser-2008}, we deduce that, 
\begin{equation*}
w_{t}^{\mathbf{\tilde{B}}}=w_{t}^{\mathbf{X}}
\end{equation*}%
where $\mathbf{X}=\left( t,\mathbf{B}_{t}\right) $, and hence, 
\begin{equation*}
L_{t}^{\mathbf{\tilde{B}}}=L_{t}^{\mathbf{X}}\text{.}
\end{equation*}%
Therefore $v$ solves,%
\begin{equation*}
\frac{\partial v}{\partial t}=L_{t}^{\mathbf{X}}v\ \ \ \ \ v\left(
0,y\right) =\phi \left( y\right)
\end{equation*}%
and since%
\begin{equation*}
\pi _{\left( V_{1},V_{2}\right) }\left( 0,y;\mathbf{\tilde{B}}\right) =\pi
_{\left( c\left[ V_{1},V_{2}\right] ,V_{1},V_{2}\right) }\left( 0,y;\left( t,%
\mathbf{B}\right) \right)
\end{equation*}%
we deduce that $u$ defined in (\ref{defn of u}) solves,%
\begin{eqnarray}
dv\left( t,y\right) &=&\left( L_{t}v\left( t,y\right) -\nabla v\left(
t,y\right) \cdot c\left[ V_{1},V_{2}\right] \left( y\right) \right)
dt-\nabla v\left( t,y\right) \cdot V\left( y\right) d\mathbf{B}_{t} \\
v\left( 0,y\right) &=&\phi \left( y\right) \text{.}  \notag
\end{eqnarray}%
From Proposition \ref{Prop BM 1}, we get that $u$ solves the Stratonovich
SPDE (\ref{SPDE with drift}).
\end{proof}

In Theorem \ref{Existence Thm 2nd order}, we saw that $\mathbf{x}\mapsto \Pi
_{\left( a,b,V\right) }\left( 0,\phi ;\mathbf{x}\right) $ is continuous as a
map from $C^{\frac{1}{p}-H\ddot{o}l}\left( \left[ 0,T\right] ,G^{\left[ p%
\right] }\left( \mathbb{R}^{d}\right) \right) $ into $C\left( \left[ 0,T%
\right] \times \mathbb{R}^{e},\mathbb{R}\right) $, with uniform topology,
whenever $V\in \mathrm{Lip}^{\gamma }\left( \mathbb{R}^{e}\right) $, $\gamma
>p+3$, and $\phi \in C_{b}^{2}\left( \mathbb{R}^{e},\mathbb{R}\right) $. It
is consistent to write $\Pi _{\left( a,b,V\right) }\left( 0,\phi ;h\right) $
for the PDE solution%
\begin{eqnarray*}
du\left( t,y\right) &=&L_{t}u\left( t,y\right) dt-\nabla u\left( t,y\right)
\cdot V\left( y\right) dh_{t} \\
u\left( 0,y\right) &=&\phi \left( y\right)
\end{eqnarray*}%
when $h\in $\ $C^{2}\left( \left[ 0,T\right] ,\mathbb{R}^{d}\right) $. \ 

\begin{theorem}[Support]
Assume $h\in C^{2}\left( \left[ 0,T\right] ,\mathbb{R}^{d}\right) $ and $%
\delta >0$. Then\footnote{%
The infinity norm of $B-h$ is based on Euclidean norm on $\mathbb{R}^{d}$.}%
\begin{equation*}
\mathbb{P}\left( \left. \left\vert \Pi _{\left( a,b,V\right) }\left( 0,\phi ;%
\mathbf{B}\right) -\Pi _{\left( a,b,V\right) }\left( 0,\phi ;h\right)
\right\vert _{\infty ;\left[ 0,T\right] \times \mathbb{R}^{e}}>\delta
\right\vert \left\vert B-h\right\vert _{\infty ;\left[ 0,T\right]
}<\varepsilon \right) \rightarrow _{\varepsilon \rightarrow 0}0\text{.}
\end{equation*}%
In particular, the topological support of the solution to the Stratonovich
SPDE (\ref{StratoSPDE}) is the closure of 
\begin{equation*}
\left\{ \Pi _{\left( a,b,V\right) }\left( 0,\phi ;h\right) :h\in C^{2}\left( 
\left[ 0,T\right] ,\mathbb{R}^{d}\right) \right\}
\end{equation*}%
in uniform topology.
\end{theorem}

\begin{proof}
The conditioning statement is a direct consequence of the main result of 
\cite{friz-lyons-stroock-06} and continuity of the RPDE solution map $\Pi
_{\left( a,b,V\right) }\left( 0,\phi ;\mathbf{\cdot }\right) $. Since $%
\left\{ \left\vert B-h\right\vert _{\infty ;\left[ 0,T\right] }<\varepsilon
\right\} $ has positive probability this implies 
\begin{equation*}
\left\{ \Pi _{\left( a,b,V\right) }\left( 0,\phi ;h\right) :h\in C^{2}\left( 
\left[ 0,T\right] ,\mathbb{R}^{d}\right) \right\} \subset \text{ \textrm{%
support}}\left( \mathbb{P}_{\ast }\Pi _{\left( a,b,V\right) }\left( 0,\phi ;%
\mathbf{B}\right) \right) \text{.}
\end{equation*}%
The other inclusion holds since%
\begin{equation*}
\mathrm{support}\left( \mathbb{P}_{\ast }\Pi _{\left( a,b,V\right) }\left(
0,\phi ;\mathbf{B}\right) \right) \subset \overline{\left\{ \Pi _{\left(
a,b,V\right) }\left( 0,\phi ;h\right) :h\in C^{\infty }\left( \left[ 0,T%
\right] ,\mathbb{R}^{d}\right) \right\} }\text{,}
\end{equation*}%
This follows directly from continuity of $\Pi _{\left( a,b,V\right) }\left(
0,\phi ;\mathbf{\cdot }\right) $, provided we can find smooth approximations 
$B^{n}$ to $B$, such that $d_{\frac{1}{p}-H\ddot{o}l;\left[ 0,T\right]
}\left( S_{2}\left( B^{n}\right) ,\mathbf{B}\right) \longrightarrow 0$. \ We
know that such approximations exist from the Karhunen-Loeve expansion of
Brownian Motion based on the $\sin /\cos $ basis of $L^{2}$, and general
results of rough path convergence of the Karhunen-Loeve expansion proved in 
\cite{friz-Victoir-2007b}.
\end{proof}

\begin{remark}
It is easy to see that the closure of $\left\{ \Pi _{\left( a,b,V\right)
}\left( 0,\phi ;h\right) :h\in C^{2}\left( \left[ 0,T\right] ,\mathbb{R}%
^{d}\right) \right\} $ coincides with the closure of $\left\{ \Pi _{\left(
a,b,V\right) }\left( 0,\phi ;h\right) :h\in W^{1,2}\left( \left[ 0,T\right] ,%
\mathbb{R}^{d}\right) \right\} $.
\end{remark}

\bigskip 

Clearly, $\Pi _{\left( a,b,V\right) }\left( 0,\phi ;\mathbf{B}\left( \sqrt{%
\varepsilon }\cdot \right) \right) $ converges in distribution as $%
\varepsilon \rightarrow 0$ to $\Pi _{\left( a,b,V\right) }\left( 0,\phi
;0\right) $, the solution of the PDE $\frac{\partial v}{\partial t}=L_{t}v$
. The following LDP\ principle quantifies the rate of this convergence. 

\begin{theorem}[Large Deviations]
The family $\left( \mathbb{P}_{\ast }\Pi _{\left( a,b,V\right) }\left(
0,\phi ;\mathbf{B}\left( \sqrt{\varepsilon }\cdot \right) \right) \right) $
satisfies a large deviation principle with good rate function%
\begin{equation*}
J\left( u\right) =\inf \left\{ \frac{1}{2}\int_{0}^{T}\left\vert \dot{h}%
_{t}\right\vert ^{2}dt:h\in W^{1,2}\left( \left[ 0,T\right] ,\mathbb{R}%
^{d}\right) \text{ and }\Pi _{\left( a,b,V\right) }\left( 0,\phi ;h\right)
=u\right\} \text{.}
\end{equation*}
\end{theorem}

\begin{proof}
One of the results proved in \cite{friz-Victoir-2005-Appli} says that the
random variables $\mathbf{B}\left( \sqrt{\varepsilon }\cdot \right) $
satisfy a large deviation principle in $\frac{1}{p}$-H\"{o}lder topology
with good rate function 
\begin{eqnarray*}
I\left( \mathbf{x}\right) &=&\frac{1}{2}\int_{0}^{T}\left\vert \dot{h}%
_{t}\right\vert ^{2}dt\ \ if\ \ S_{2}\left( h\right) =\mathbf{x\ \ }for\
some\ h\in W^{1,2}\left( \left[ 0,T\right] ,\mathbb{R}^{d}\right) \\
&=&+\infty \ \ \ \ otherwise\text{.}
\end{eqnarray*}%
Using the continuity of $\Pi _{\left( a,b,V\right) }\left( 0,\phi ;\mathbf{%
\cdot }\right) $ and the contraction principle, the required large deviation
principle for $\left( \mathbb{P}_{\ast }\Pi _{\left( a,b,V\right) }\left(
0,\phi ;\mathbf{B}\left( \sqrt{\varepsilon }\cdot \right) \right) \right) $
follows immediately.
\end{proof}

\section{SPDEs with Markovian noise\label{Section Markov RPDEs}}

Let $X$ be a Markov process with uniformly elliptic generator in divergence
form (c.f. \cite{St88}). The coefficient matrix in the generator need not
have any regularity (beyond measurability), in which case $X$ is not a
semi-martingale\footnote{%
Nonetheless, sample paths properties of $X$ are very similar to those of
Brownian motion.}. Stochastic area cannot be defined via iterated stochastic
integrals but there are alternative constructions (\cite{lyons-stoica-99}, 
\cite{LejI06}, \cite{friz-victoir-2008-Markov}) that lift $X$ to a
"Markovian" rough path $\mathbf{X}\in C^{\frac{1}{p}-H\ddot{o}l}\left( \left[
0,T\right] ,G^{2}\left( \mathbb{R}^{d}\right) \right) $ for any $p\in \left(
2,3\right) $. With the RPDE approach, we can then give a meaning to the SPDE%
\footnote{%
We assume that the coefficients $a$ and $b$, together with the vector field $%
V$ have enough regularity, (namely the assumptions made at the beginning of
Section \ref{Section Brownian RPDES}) for the RPDE to have unique solutions.}%
\begin{eqnarray}
du\left( t,y\right) &=&L_{t}u\left( t,y\right) dt-\nabla u\left( t,y\right)
\cdot V\left( y\right) d\mathbf{X}_{t}  \label{MarkovSPDE} \\
\,\,u\left( 0,y\right) &=&\phi \left( y\right)  \notag
\end{eqnarray}%
which generalizes Stratonovich SPDEs to "SPDEs with (uniformly elliptic)
Markovian noise". Various convergence results proved in \cite%
{friz-victoir-2008-Markov} together with our RPDE continuity result, give an
appealing probabilistic meaning to such SPDE solutions. For instance, if the
coefficient matrix is mollified (with parameter $\varepsilon $) so that $%
X^{\varepsilon }$ is a semi-martingale, one constructs without difficulties
(c.f. \cite{Kun-90} ) a Stratonovich solution to%
\begin{eqnarray*}
du^{\varepsilon }\left( t,y\right) &=&L_{t}u^{\varepsilon }\left( t,y\right)
dt-\nabla u^{\varepsilon }\left( t,y\right) \cdot V\left( y\right) \circ
dX_{t}^{\varepsilon } \\
\,\,u^{\varepsilon }\left( 0,y\right) &=&\phi \left( y\right)
\end{eqnarray*}%
and as $\varepsilon \rightarrow 0$, the solution $u^{\varepsilon }$
converges in distribution to the solution of (\ref{MarkovSPDE}). Similarly,
if $X$ is replaced by a piecewise linear approximation $X^{n}$, we can solve%
\begin{eqnarray*}
du^{n}\left( t,y\right) &=&L_{t}u^{n}\left( t,y\right) dt-\nabla u^{n}\left(
t,y\right) \cdot V\left( y\right) \circ dX_{t}^{n} \\
\,\,u^{n}\left( 0,y\right) &=&\phi \left( y\right)
\end{eqnarray*}%
as (time-inhomogenous) linear second order PDE and as $n\rightarrow \infty $
we have convergence (in probability) to the solution of (\ref{MarkovSPDE}).
Support and large deviation properties for Markovian rough paths\ were
established in \cite{friz-victoir-2008-Markov} and similar reasoning as in
the Brownian case leads to support and large deviation statements for these
SPDEs with Markovian noise. The details are straight-forward and omitted.

\section{SPDEs with Gaussian noise\label{Section Gaussian RPDEs}}

Let $X=\left( X^{1},\ldots X^{d}\right) $ be a continuous centred Gaussian
process with independent components started at zero, and suppose that its
covariation $R^{X}$, has finite $\rho $-variation (in $2D$-sense) with $\rho
\in \left[ 1,2\right) $, bounded by a H\"{o}lder dominated control\footnote{%
A $2D$ control $\omega $ is H\"{o}lder dominated if there exists a constant $%
C$ such that for all $0\leq s<t\leq T$, $\omega \left( \left[ s,t\right]
^{2}\right) \leq C\left\vert t-s\right\vert $. \ In particular, this implies
that $R_{\rho -var;\left[ s,t\right] ^{2}}^{X}\leq C\left\vert
t-s\right\vert ^{\frac{1}{\rho }}$.}. \ Then from \cite{friz-Victoir-2007a},
we know that for $p\in \left( 2\rho ,4\right) $, $X$ lifts to a geometric H%
\"{o}lder $p$-rough path $\mathbf{X=X}\left( \omega \right) $, a "Gaussian
rough path". With $\mathrm{Lip}^{\gamma }$-vector fields $V=\left(
V_{1},\dots ,V_{d}\right) $, $\gamma >p+3$, and $\phi \in C_{b}^{2}\left( 
\mathbb{R}^{e};\mathbb{R}\right) $, the RPDE

\begin{eqnarray}
du\left( t,y\right) &=&L_{t}u\left( t,y\right) dt-\nabla u\left( t,y\right)
\cdot V\left( y\right) d\mathbf{X}_{t}  \label{pde with gaussian noise} \\
u\left( 0,y\right) &=&\phi \left( y\right)  \notag
\end{eqnarray}%
can be solved for almost every $\omega $ and has the obvious interpretation
of an SPDE with Gaussian noise. (The setup of \cite{friz-Victoir-2007a}
includes (multi-dimensional) Brownian motion with $\rho =1$, fractional
Brownian motion with $\rho =1/\left( 2H\right) $ for $H\in \left(
1/4,1/2\right) $, the case $H>1/2$ being trivial, the Ornstein-Uhlenbeck
process, and the Brownian bridge process, among many other examples).

There is an equivalent statement for most of what has been said in Section %
\ref{Section Markov RPDEs}: various weak and strong approximation results
make the interpretation of the solution to (\ref{pde with gaussian noise})
easy. Replacing $X$ by piecewise linear approximations $X^{n}$ (or mollifier
approximations $X^{\delta }$) reduces (\ref{pde with gaussian noise}) to a
(time-inhomogenous) linear second order PDE, and as $n\rightarrow \infty $
(resp. $\delta \rightarrow 0$), these solutions converge (in probability) to
the solution of (\ref{pde with gaussian noise}).

There is a support result for such Gaussian rough paths (always in the
appropriate $1/p$-H\"{o}lder rough paths topology c.f. \cite%
{friz-Victoir-2007b}) and with the continuity of $\mathbf{X}\mapsto \Pi
_{\left( a,b,V\right) }\left( 0,\phi ;\mathbf{X}\right) $, the solution map
to (\ref{pde with gaussian noise}), we immediately get that the support of
the law of $\Pi _{\left( a,b,V\right) }\left( 0,\phi ;\mathbf{X}\right) $,
in uniform topology, is the closure of all second order PDE solutions $\Pi
_{\left( a,b,V\right) }\left( 0,\phi ;h\right) $ where $h$ $\in \mathcal{H}$%
, the Cameron Martin space associated with $X$.

\begin{remark}
Thanks to the known embedding\ $\mathcal{H}\hookrightarrow C^{\rho \text{-var%
}}\left( \left[ 0,T\right] ,\mathbb{R}^{d}\right) $, $\rho \in \left[
1,2\right) $ (cf. \cite{friz-Victoir-2007a}), we can define $S_{\left[ p%
\right] }\left( h\right) $, $p\in \left( 2\rho ,4\right) $, using Young
integration. \ If we denote $S_{\left[ p\right] }\left( h\right) $ by $%
\mathbf{h}$, then,%
\begin{eqnarray*}
\left\Vert \mathbf{h}_{s,t}\right\Vert &\leq &\left\vert \mathbf{h}%
\right\vert _{\rho -var;\left[ s,t\right] }\leq K_{1}\left\vert h\right\vert
_{\rho -var;\left[ s,t\right] }\ \ (cf.\cite{lyons-98}\ \text{Thm. }2.2.1) \\
&\leq &K_{1}\left\vert h\right\vert _{\mathcal{H}}\sqrt{R_{\rho -var;\left[
s,t\right] ^{2}}^{X}}\ \ \ \ (cf.\cite{friz-Victoir-2007a}\ \text{Prop. }16)
\\
&\leq &K_{2}\left\vert h\right\vert _{\mathcal{H}}\left\vert t-s\right\vert
^{\frac{1}{2\rho }} \\
&\leq &K_{3}\left\vert t-s\right\vert ^{\frac{1}{p}}
\end{eqnarray*}%
and thus $\mathbf{h}\in C^{\frac{1}{p}-H\ddot{o}l}\left( \left[ 0,T\right]
,G^{\left[ p\right] }\left( \mathbb{R}^{d}\right) \right) $. \ Therefore,
when we refer to $\Pi _{\left( a,b,V\right) }\left( 0,\phi ;h\right) $, we
are basically considering $\Pi _{\left( a,b,V\right) }\left( 0,\phi ;\mathbf{%
h}\right) =\Pi _{\left( a,b,V\right) }\left( 0,\phi ;S_{\left[ p\right]
}\left( h\right) \right) $, which we know exists from Theorem \ref{Existence
Thm 2nd order}. \ By a basic theorem in rough path theory (Theorem 1 \cite%
{lyons-98}) $h\longrightarrow \Pi _{\left( a,b,V\right) }\left( 0,\phi ;S_{%
\left[ p\right] }\left( h\right) \right) $ is continuous in $\rho $%
-variation and thus here we are dealing with a Young PDE.
\end{remark}

There is also a LDP for $\left( \delta _{\varepsilon }\mathbf{X}:\varepsilon
>0\right) $ where $\delta _{\varepsilon }$ is the dilation operator which
generalizes scalar multiplication on $\mathbb{R}^{d}$ to $G^{\left[ p\right]
}\left( \mathbb{R}^{d}\right) $, $p\in \left( 2\rho ,\gamma \right) $ (c.f. 
\cite{friz-Victoir-2007}). \ Keeping $u^{\varepsilon }\left( 0,y\right)
=\phi \left( y\right) $ for all $\varepsilon >0$, we abuse notation and
write 
\begin{equation*}
du^{\varepsilon }\left( t,y\right) =L_{t}u^{\varepsilon }\left( t,y\right)
dt-\varepsilon \nabla u^{\varepsilon }\left( t,y\right) \cdot V\left(
y\right) d\mathbf{X}_{t}
\end{equation*}%
rather than%
\begin{equation*}
du^{\varepsilon }\left( t,y\right) =L_{t}u^{\varepsilon }\left( t,y\right)
dt-\nabla u^{\varepsilon }\left( t,y\right) \cdot V\left( y\right) d\left(
\delta _{\varepsilon }\mathbf{X}\right) _{t}\text{.}
\end{equation*}%
Then the laws of $u^{\varepsilon }\left( t,y;\omega \right) $ satisfy a LDP
(in uniform topology) with good rate function%
\begin{equation*}
J\left( u\right) =\inf \left\{ \frac{1}{2}\left\vert h\right\vert _{\mathcal{%
H}}^{2}:h\in \mathcal{H}\text{ and }\Pi _{\left( V\right) }\left( 0,\phi
;h\right) =u\right\} \text{.}
\end{equation*}

\subsection{Density result for non-degenerate first order SPDEs with
Gaussian noise}

We now discuss whether the solution of the first order SPDE 
\begin{eqnarray}
du\left( t,y\right) +\nabla u\left( t,y\right) \cdot V\left( y\right) d%
\mathbf{X}_{t} &=&0  \label{first order gaussian rpde} \\
u\left( 0,y\right) &=&\phi \left( y\right)  \notag
\end{eqnarray}%
at some fixed point in time-space i.e. $\Pi _{\left( V\right) }\left( 0,\phi
;\mathbf{X}\left( \omega \right) \right) \left( t,y\right) =\phi \left( \pi
_{\left( V\right) }\left( 0,y;\overleftarrow{\mathbf{X}\left( \omega \right) 
}^{t}\right) _{t}\right) $, admits a density with respect to Lebesgue
measure. The question obviously reduces to establishing a density for $\pi
_{\left( V\right) }\left( 0,y;\overleftarrow{\mathbf{X}}^{t}\right) _{t}$
and then imposing the necessary non-degeneracy conditions on $\phi $. The
existence of a density for the solution of an RDE driven by a Gaussian
signal was proved in \cite{CASS-FRIZ-VICTOIR-2007} under the following
assumptions on the vector fields and the driving signal.

\begin{condition}[Ellipticity assumption on the vector fields]
\label{Ellipticity condition}The vector fields $V_{1},\ldots ,V_{d}$ span
the tangent space\ at $y$.
\end{condition}

\begin{condition}[{Non-degeneracy of the Gaussian process on $\left[ 0,T%
\right] $}]
\label{Non degeneracy condition on gaussian signal}Fix $T>0$. \ We assume
that for any smooth $f=\left( f_{1}\,,\ldots ,f_{d}\right) :\left[ 0,T\right]
\longrightarrow \mathbb{R}^{d}$,%
\begin{equation*}
\left( \int_{0}^{T}fdh\equiv \sum_{k=l}^{d}\int_{0}^{T}f_{k}dh^{k}=0\ \
\forall h\ \in \mathcal{H}\right) \Longrightarrow f\equiv 0
\end{equation*}%
where $\mathcal{H}$ is the Cameron Martin space associated with the Gaussian
process.
\end{condition}

As remarked in the same paper, non-degeneracy on $\left[ 0,T\right] $
implies non-degeneracy on $\left[ 0,t\right] $ for any $t\in \left( 0,T%
\right] $. \ We then have the following theorem.

\begin{theorem}
\label{Existence of density for solution of rde}(c.f. \cite%
{CASS-FRIZ-VICTOIR-2007}) Let $\mathbf{X}$ be a natural lift of a
continuous, centered Gaussian process with independent components $X=\left(
X_{1},\ldots ,X_{d}\right) $, with finite $\rho \in \left[ 1,2\right) $%
-variation of the covariance, bounded by a H\"{o}lder dominated control, and
non-degenerate in the sense of Condition \ref{Non degeneracy condition on
gaussian signal}. \ Let $V=\left( V_{1},\ldots ,V_{d}\right) $ be a
collection of $\mathrm{Lip}^{\gamma }$-vector fields on $\mathbb{R}%
^{e},\gamma >2\rho $, which satisfy the ellipticity Condition \ref%
{Ellipticity condition}. \ Then the solution of the rough differential
equation,%
\begin{equation*}
dY_{t}=V\left( Y_{t}\right) d\mathbf{X}_{t}\ \ \ \ Y_{0}=y
\end{equation*}%
admits a density at all times $t\in \left( 0,T\right] $ with respect to
Lebesgue measure on $\mathbb{R}^{e}$.
\end{theorem}

Using Theorem \ref{Existence Theorem} and the above, we can prove the
following result on the existence of a density for the solution of a RPDE.

\begin{theorem}
Let $\mathbf{X}$ be a natural lift of a continuous, centered Gaussian
process with independent components $X=\left( X_{1},\ldots ,X_{d}\right) $,
with finite $\rho \in \left[ 1,2\right) $-variation of the covariance,
bounded by a H\"{o}lder dominated control, and non-degenerate in the sense
of Condition \ref{Non degeneracy condition on gaussian signal}. Let $%
V=\left( V_{1},\ldots ,V_{d}\right) $ be a collection of $\mathrm{Lip}%
^{\gamma }\left( \mathbb{R}^{e}\right) $ vector fields, $\gamma >2\rho $,
and suppose\footnote{%
Cf. remark \ref{phi_unbd}.} that $\phi \in C^{1}\left( \mathbb{R}^{e};%
\mathbb{R}\right) $ is non-degenerate, i.e. $\nabla \phi \neq 0$ everywhere.
With $\mathbf{X=X}\left( \omega \right) $, the solution $u\left( t,y\right)
=u\left( t,y;\omega \right) $ to the random RPDE%
\begin{eqnarray}
du\left( t,y\right) +\nabla u\left( t,y\right) \cdot V\left( y\right) d%
\mathbf{X}_{t} &=&0 \\
u\left( 0,y\right)  &=&\phi \left( y\right)   \notag
\end{eqnarray}%
has a density with respect to the Lebesgue measure on $\mathbb{R}$,\ for
each $t\in \left( 0,T\right] $ and for each $y\in \mathbb{R}^{e}$ for which
Condition \ref{Ellipticity condition} holds.
\end{theorem}

\begin{proof}
Fix $t\in \left( 0,T\right] $, and choose $y\in \mathbb{R}^{e}$ such that
the vector fields $V_{1},\ldots ,V_{d}$ span the tangent space at $y$. \ We
have to show that 
\begin{equation*}
u\left( t,y\right) =\phi \left( \pi _{\left( V\right) }\left( 0,y;%
\overleftarrow{\mathbf{X}}^{t}\right) _{t}\right) 
\end{equation*}%
has a density. \ 

We first note that $\overleftarrow{\mathbf{X}}_{\cdot }^{t}$ is again a
Gaussian geometric $p$-rough path, $p\in \left( 2\rho ,4\right) $, defined
on $\left[ 0,t\right] $. \ We want to prove that $\overleftarrow{X}^{t}$
satisfies the non-degeneracy Condition \ref{Non degeneracy condition on
gaussian signal} on $\left[ 0,t\right] $. \ Let $f$ be a smooth function and
suppose that $\int fdg\equiv 0$ for all $g\in \mathcal{G}$, the Cameron
Martin space associated with $\overleftarrow{X}^{t}$. \ Recall that elements
of $\mathcal{G}$ are of the form, $g_{s}=\mathbb{E}\left( \overleftarrow{X}%
_{s}^{t}\xi \left( g\right) \right) $ where $\xi \left( g\right) $ is a
Gaussian random variable. \ For $s\in \left[ 0,t\right] $,%
\begin{equation*}
g_{s}=\mathbb{E}\left( \overleftarrow{X}_{s}^{t}\xi \left( g\right) \right) =%
\mathbb{E}\left( X_{t-s}\xi \left( g\right) \right) =h_{t-s}
\end{equation*}%
for some $h\in \mathcal{H}$, the Cameron Martin space associated with $X$. \
Thus 
\begin{equation*}
\left( \int_{0}^{t}f_{s}dg_{s}\equiv 0\ \ \forall g\in \mathcal{G}\right) 
\mathcal{\Leftrightarrow }\left( \int_{0}^{t}f_{s}dh_{t-s}\equiv 0\ \forall
h\in \mathcal{H}\right) \text{.}
\end{equation*}%
Since $f$ is smooth, the above integrals are Riemann-Stieltjes integrals,
and so, using a simple change of variable, we get that 
\begin{equation*}
\int_{0}^{t}f_{s}dh_{t-s}=-\int_{0}^{t}f_{t-s}dh_{s}\equiv 0\ \forall h\in 
\mathcal{H}\text{.}
\end{equation*}%
But $-f_{t-\cdot }$ is of course a smooth function, and hence it follows
from the non-degeneracy condition on $X$, that $-f_{t-\cdot }\equiv 0$. \
This implies that $f\equiv 0$. \ 

Therefore the Gaussian process $\overleftarrow{X}_{\cdot }^{t}$ also
satisfies the non-degeneracy Condition \ref{Non degeneracy condition on
gaussian signal}, and hence we can deduce from our choice of $y\in \mathbb{R}%
^{e}$ and Theorem \ref{Existence of density for solution of rde}, that the
random variable $\pi _{\left( V\right) }\left( 0,y;\overleftarrow{\mathbf{X}}%
^{t}\right) _{t}$ has a density with respect to the Lebesgue measure on $%
\mathbb{R}^{e}$.

From the non-degeneracy assumption on the initial function $\phi $, the
existence of a density for $u\left( t,y\right) $ now follows immediately.
\end{proof}

\begin{acknowledgement}
The authors would like to thank M. Hairer for some feedback on the final
version of this work.
\end{acknowledgement}

\bibliographystyle{plain}
\bibliography{acompat,rough3}

\end{document}